\documentclass{ws-m3as}
\usepackage[english]{babel}
\usepackage{comment}

\usepackage{amsmath,amssymb}
\usepackage{graphicx}
\usepackage{tikz}
\usetikzlibrary{calc}
\usepackage[dvipsnames]{xcolor}
\usepackage[colorlinks=true, allcolors=blue]{hyperref}

\usepackage{lineno}
\usepackage{mathtools}
\usepackage{pgfplots}
\pgfplotsset{compat=1.15}
\usepackage{mathrsfs}
\usetikzlibrary{arrows}

\DeclareFontFamily{OT1}{pzc}{}
\DeclareFontShape{OT1}{pzc}{m}{it}{<-> s * [1.10] pzcmi7t}{}
\DeclareMathAlphabet{\smallcal}{OT1}{pzc}{m}{it}
\newcommand{\physH}{\smallcal{h}}
\newcommand{\elemSet}{\mathcal{E}}

\usepackage{framed}
\usepackage{subfiles}

\newcommand{\neighbori}{\hat{I}}
\newcommand{\neighborj}{\hat{J}}

\newcommand{\tpprojector}{\widehat{\Pi}^{(r,r)}_{p,k,\mathbf Z}}

\newcommand{\tpprojectorS}{\widehat{\Pi}_{AS}}
\newcommand{\tpprojectorSph}{\Pi_{AS}}

\newcommand{\tpprojectorsimple}{\mathcal{Q}}
\newcommand{\tpprojectorsimpleI}[1]{\tpprojectorsimple_{#1}}

\newcommand{\edgeextensionS}[2]{E_{#1}^{(#2)}} 
\newcommand{\edgeextension}[3]{\edgeextensionS{#1}{#2}\left( #3 \right)}

\newcommand{\edgeprojectorS}[2]{\Pi_{#1}^{(#2)}}

\newcommand{\vardelta}{\delta}

\DeclareMathOperator*{\essinf}{ess.\hspace{-.2em}inf}

\newcommand{\citeRefs}[1]{Refs.~\refcite{#1}}
\newcommand{\citeOneRef}[1]{Ref.~\refcite{#1}}
\newcommand{\citeRefDetail}[2]{Ref.~\refcite{#2}, #1}

\title{Robust approximation error estimates for analysis-suitable~$G^1$ isogeometric multi-patch discretizations}

\author{Fatima Hasanova}
\address{Johann Radon Institute for Computational and Applied Mathematics,\\
Austrian Academy of Sciences,\\
Altenberger Str. 69, 4040 Linz, Austria\\
fatima.hasanova@ricam.oeaw.ac.at}

\author{Stefan Takacs}

\address{Johannes Kepler University,\\
Altenberger Str. 69, 4040 Linz, Austria\\
stefan.takacs@numa.uni-linz.ac.at}

\author{Thomas Takacs}

\address{Johann Radon Institute for Computational and Applied Mathematics,\\
Austrian Academy of Sciences,\\
Altenberger Str. 69, 4040 Linz, Austria\\
thomas.takacs@ricam.oeaw.ac.at}

\begin{document}
\maketitle

\begin{abstract}
We prove $p$-robust approximation error estimates for $H^2$-conforming isogeometric discretizations over planar multi-patch domains. Possible applications are fourth order boundary value problems, like the biharmonic equation or Kirchhoff–Love plates. Using Isogeometric Analysis, such conforming discretizations can be constructed effortlessly for the single-patch case. In order to obtain a globally $H^2$-conforming discretization in the multi-patch case, the functions must be $C^1$-smooth across the interfaces between the patches. To obtain optimal approximation properties, those $C^1$-smooth spaces must also reproduce splines of sufficiently high degree for traces and transversal derivatives at all patch interfaces. Such constructions are based on some assumptions on the geometry. We restrict ourselves to the class of analysis-suitable $G^1$ (AS-$G^1$) multi-patch domains, which is the subset of $C^0$-matching multi-patch domains that allows the definition of spline spaces that yield the necessary reproduction properties without the need to locally increase the degree. While approximation error estimates have been established for single-patch and $C^0$ isogeometric multi-patch spaces, corresponding results for the $C^1$ multi-patch setting have been missing. The resulting bounds on the approximation error depend on the geometry parameterization and on the Sobolev regularity of the target function, but are independent of the spline degree $p$.
\end{abstract}

\section{Introduction}\label{sec:Intro}

We provide approximation error estimates for $C^1$-smooth isogeometric multi-patch discretizations. Since the resulting function spaces are $H^2$-conforming, possible applications are conforming Galerkin discretizations of fourth-order partial differential equations, such as the biharmonic equation, Kirchhoff--Love plates and shells, cf. \citeRefs{kiendl2009isogeometric,kiendl2010bending} and phase-field formulations like the Cahn--Hilliard equation, cf. \citeOneRef{gomez2008isogeometric}. Since Isogeometric Analysis (IGA), cf. \citeRefs{hughes2005isogeometric,cottrell2009isogeometric}, uses spline-based constructions, arbitrarily smooth spaces can be constructed effortlessly. Usually, non-trivial geometries cannot be represented using a single smooth parameterization as required for (single-patch) IGA. One of the most common generalizations is multi-patch IGA, where the computational domain is subdivided into separately parameterized patches, over which function spaces can be defined.

Following this framework, we consider a $C^0$-conforming multi-patch isogeometric space and study its $C^1$-smooth subspace. The $C^1$-smoothness of an isogeometric function is equivalent to the geometric continuity of its graph, cf. \citeOneRef{groisser2015matched}. Consequently, the $C^1$ subspace can be characterized algebraically by these continuity conditions, as done in \citeOneRef{mourrain2016dimension}. While the dimension and the basis of a $C^0$ multi-patch isogeometric space of degree $p$ and patch-internal smoothness $k$ over a $C^0$-matching multi-patch parameterization are well-understood and independent of the parameterization of the geometry, the dimension of the $C^1$-smooth subspace depends in a non-trivial and non-robust way on the parameterization.

In this paper, we use the concept of analysis-suitable $G^1$ (AS-$G^1$) smoothness, which was introduced in \citeOneRef{CST16} and allows one to define a $C^1$-smooth space of parameterization-independent dimension. Related approaches were established in parallel in \citeRefs{kapl2015isogeometric,nguyen2016c1} and \refcite{bercovier2017smooth}; the latter is based on \citeOneRef{matskewich2001construction}. Restricting the geometry to such AS-$G^1$ parameterizations is not a significant limitation, since all planar domains that have a boundary composed of spline curves possess an AS-$G^1$ parameterization, see \citeOneRef{kapl2018construction}. Moreover, the full $C^1$-smooth subspace of a generic $C^0$ multi-patch space generally does not reproduce splines of sufficiently high degree for traces and/or crossing derivatives along patch interfaces and therefore lacks optimal approximation properties unless the AS-$G^1$ condition is satisfied. While dimension analysis, basis constructions, and AS-$G^1$ reparameterization strategies have already been developed in \citeRefs{Kapl2017,kapl2019isogeometric,kapl2018construction}, repsectively, and numerical evidence suggests optimal convergence, theoretical approximation results have been missing.

We fill this gap in the theory by proving robust approximation error estimates for AS-$G^1$ isogeometric multi-patch discretizations.
For the single-patch case, the first approximation error estimate which is robust in the degree $p$ and spline regularity $k$, was developed in \citeOneRef{beirao2011some} for the case $p\geq 2k-1$. Since then, general $p$-robust estimates have been developed for univariate and tensor-product B-splines in \citeRefs{TT16,floater2017optimal,floater2019optimal,S19,bressan2019approximation,sande2022ritz} and extended to $C^0$ multi-patch splines in \citeOneRef{T18}. In a finite element context, related estimates for (bilinear) $C^1$-elements were established in \citeOneRef{Kapl2021}. We extend these results to the AS-$G^1$ multi-patch setting, for which we construct a global projector. The global projector is defined patch-wise. On each patch, it is defined as a modified tensor-product projector built from $H^2$-orthogonal univariate projectors. The modification is done by introducing correction terms that yield matching traces and directional derivatives along all interfaces. This projector yields error estimates that are robust in the sense that the constants depend on the geometry parameterization and on the Sobolev regularity, but are independent of the spline degree~$p$. All bounds are given in terms of classical and bent Sobolev norms.

The theory developed in this paper applies directly to the standard AS-$G^1$ setting as defined in \citeOneRef{CST16} and as used in, e.g., \citeRefs{kapl2015isogeometric,Kapl2017isogeometric,Kapl2017,kapl2018construction,kapl2019smai,kapl2019isogeometric}. However, the methods presented here are also relevant to a wider range of related approaches with varying degrees of modification. In \citeOneRef{benvenuti2023isogeometric}, the same patch-wise spaces are considered and coupling is performed using mortaring, where parts of the theory can be applied. Our theory applies with minor modifications to constructions over multi-patch surfaces, cf. \citeRefs{Farahat2023a,Farahat2023b,Farahat2024}, where Sobolev regularity of functions over surfaces must be defined in such a way that the patch-wise pullbacks possess the same regularity properties as in the planar case. Furthermore, the projector can be modified to yield local error estimates, which are needed for adaptivity, as explored in \citeRefs{bracco2020isogeometric,Bracco2023,Bracco2024}, where hierarchical constructions were employed.

For approaches that require higher regularity, such as  \citeRefs{kapl2017space,kapl2017spaceB,kapl2018dimension,kapl2019solving,kapl2020isogeometric,Kapl2021c_adv,kapl2024isogeometric,Kapl2026isogeometric,Kapl2026}, a similar projector can be defined using higher order correction terms. The framework developed in this paper may also be extended to $C^1$-spaces over mixed triangular and quadrilateral B\'ezier meshes, as in \citeRefs{grovselj2020super,Groselj2024}, to $C^1$-spaces over volumetric domains, as in \citeRefs{birner2018approximation,birner2018bases,birner2019space,kapl2022c1}, or to strong $C^1$-coupling strategies over multi-patch domains, as in \citeRefs{chan2018isogeometric,Chan2019strong}. Related isogeometric discretizations based on $G^1$ splines, cf. \citeRefs{nguyen2016c1,toshniwal2017smooth,blidia2017g1,blidia2020geometrically,marsala2022g1,marsala2024g1,marsala2024cad}, on approximate $C^1$ splines as in \citeRefs{Weinmuller2021,Weinmuller2022,seiler2018approximately} or on almost-$C^1$ splines over unstructured meshes as in \citeRefs{Takacs2023,wen2023isogeometric} may also be analyzed using similar approaches. Additionally, our analysis may also be relevant to specific formulations for Kirchhoff plate models, cf. \citeRefs{greco2019quadrilateral,horger2019hybrid}. 
Recently, various smooth and approximately smooth basis constructions were compared in detail in the two overview papers \citeOneRef{hughes2021smooth} and \citeOneRef{Verhelst2024}, highlighting their practical utility in computational mechanics.

While the presented approximation estimates yield optimal convergence rates in $h$, the fact that they are robust in $p$ allows the development of iterative solvers that exhibit convergence rates independent of the mesh size $h$ and the spline degree $p$. This includes robust multigrid methods for multi-patch discretizations of the biharmonic equation, cf. \citeOneRef{T18} on multigrid for second order PDEs over multi-patch domains and \citeOneRef{sogn2023multigrid} on multigrid for the biharmonic equation over a single patch, as well as domain decomposition techniques such as the Isogeometric Tearing and Interconnecting (IETI-DP) methods recently proposed for high-order problems over multi-patch geometries in \citeOneRef{Takacs2025isogeometric}.

The remainder of the paper is organized as follows. Section \ref{sec:AS-G1 discr} introduces the model problem, the AS-$G^1$ multi-patch parameterization and the corresponding isogeometric discretization spaces. Section \ref{sec:approx-error-estimates} is dedicated to the construction of the projector and the derivation of the approximation error estimates. We conclude the paper in Section \ref{sec:conclusions}.

\section{Analysis-suitable $G^1$ isogeometric multi-patch discretizations} \label{sec:AS-G1 discr}

In this section we introduce the model problem, the class of multi-patch geometries we consider, and the isogeometric spaces over those geometries. Specifically, those multi-patch geometries follow the definition of analysis-suitable $G^1$ smoothness as introduced in \citeOneRef{CST16}.

\subsection{Model problem}\label{subsec:model-problem}

The application that motivated the definition and development of analysis-suitable $G^1$ discretizations is the conforming discretization of fourth order boundary value problems.
We consider the first biharmonic problem as the simplest model problem.
Let $\Omega\subset\mathbb R^2$ be an open, bounded and connected domain with Lipschitz boundary $\partial\Omega$. The first biharmonic problem reads as follows. Given a sufficiently smooth function $f:\Omega\to \mathbb R$,
find a sufficiently smooth function $u:\Omega\to \mathbb R$ such that
\begin{equation}\label{eq:model-problem-strong}
    \Delta^2 u = f \quad\mbox{in}\quad\Omega,
    \qquad u=\partial_n u= 0 \quad\mbox{on}\quad \partial\Omega,
\end{equation}
where $\partial_n$ denotes the directional derivative in direction of the normal vector.
The weak formulation of the problem is formulated in the Sobolev space $H^2_0(\Omega)$. Here and in what follows, we use standard Sobolev spaces as defined in Ref.~\refcite{Adams} or \refcite{Lions1972}. 
We denote by $L^2(\Omega)$ the Lebesgue space of square integrable functions, by $L^\infty(\Omega)$ the space of almost everywhere bounded functions, and by $H^k(\Omega)=W^{k,2}(\Omega)$ and $W^{k,\infty}(\Omega)$, $k\ge 0$ integer, the Sobolev space of $k$ times weakly differentiable functions, where those weak derivatives are in $L^2(\Omega)$ or $L^\infty(\Omega)$, respectively. Following this notation, we have $H^0(\Omega)=L^2(\Omega)$ and $W^{0,\infty}(\Omega)=L^\infty(\Omega)$. All of these spaces are equipped with the standard norms $\|\cdot\|_{L^2(\Omega)}$, $\|\cdot\|_{L^\infty(\Omega)}$, $\|\cdot\|_{H^k(\Omega)}$, $\|\cdot\|_{W^{k,\infty}(\Omega)}$, semi norms $|\cdot|_{H^k(\Omega)}$ and the standard scalar product $(\cdot,\cdot)_{L^2(\Omega)}$. 
Restrictions of functions in those Sobolev spaces to boundaries and interfaces are to be interpreted in the sense of the trace operator. 
$H^k_0(\Omega)\subset H^k(\Omega)$ denotes all functions whose function values and derivatives up to order $k-1$ vanish on $\partial\Omega$.
The same notation is used for subdomains of $\Omega$ and (parts) of their boundary.

The variational formulation of~\eqref{eq:model-problem-strong} reads as follows. 
Given $f\in L^2(\Omega)$, find $u\in H^2_0(\Omega)$ such that
\begin{equation}\label{eq:model-problem}
    \underbrace{\int_\Omega \Delta u\, \Delta v \, \mathrm d x}_{\displaystyle a(u,v):=}
    =
    \underbrace{\int_\Omega f\, v \, \mathrm d x}_{\displaystyle (f,v)_{L^2(\Omega)} =}
    \quad
    \mbox{for all}
    \quad
    v \in H^2_0(\Omega).
\end{equation}
Since $\|\Delta u\|_{L^2(\Omega)} = |u|_{H^2(\Omega)}$ for all $u\in H^2_0(\Omega)$ (see \citeRefDetail{Chapter~VI, \S5}{Braess}) and using a standard Poincaré type inequality, we obtain that the bilinear form $a(\cdot,\cdot)$
is elliptic and bounded, i.e., there are constants $\mu_1,\mu_2>0$ such that
\begin{equation*}
    a(u,u)\ge \mu_1 \|u\|_{H^2(\Omega)}^2, \qquad
    a(u,v)\le \mu_2 \|u\|_{H^2(\Omega)}\|v\|_{H^2(\Omega)}
    \quad\mbox{for all}\quad u,v\in H^2_0(\Omega).
\end{equation*}
So, the theorem of Lax--Milgram (cf.~\citeRefDetail{Chapter~II, Theorem~2.5}{Braess}) guarantees existence and uniqueness of a solution to~\eqref{eq:model-problem}.

For discretization, any finite-dimensional subspace $\mathcal V_h \subset H^2_0(\Omega)$ can be chosen, like an analysis-suitable $G^1$ isogeometric multi-patch discretization, that we study in this paper. No matter which space was chosen, the Galerkin approximation is the function $u_h\in \mathcal V_h$ such that
\begin{equation*}
    a( u_h\,  v_h ) 
    =
    (f,v_h)_{L^2(\Omega)}
    \quad
    \mbox{for all}
    \quad
    v_h \in \mathcal V_h.
\end{equation*}
The discretization error can be estimated using Céa's lemma as
\begin{equation*}
    \|u-u_h\|_{H^2(\Omega)}
    \le
    \frac{\mu_2}{\mu_1} 
    \inf_{v_h \in \mathcal V_h} \|u-v_h\|_{H^2(\Omega)}.
\end{equation*}
The approximation error can be further estimated using error estimators for the quasi-interpolator that we construct in this paper. Depending on the smoothness of the solution and the chosen spline degrees, one would expect error estimates like
\begin{equation}\label{eq:hs bound}
    \|u-u_h\|_{H^2(\Omega)}
    \le
    C \physH^{s-2} \|u\|_{H^{s}(\Omega)}
\end{equation}
for some $s>2$ and some constant $C>0$. Here $\physH$ is the mesh size corresponding to $\mathcal V_h$. A formal definition is given in~\eqref{eq:phys h}.

In the following, we discuss how to construct the function space $\mathcal V_h$ using Isogeometric Analysis, starting with the representation of the geometry in the next subsection and the construction of the function space itself in the then following subsection.

\subsection{Representation of the geometry}\label{subsec:geometry}

We assume that $\Omega$ is composed of multiple non-overlapping open patches $\Omega^{(i)}$, $i \in \mathcal{I}_{\Omega}$, where the index set $\mathcal I_\Omega$ is a finite set of integers, such that
\begin{equation*}
    \overline\Omega = \bigcup_{i\in \mathcal{I}_{\Omega}} \overline{\Omega^{(i)}},
\end{equation*}
with each patch parameterized via a bijective mapping
\begin{equation*}
\mathbf{G}^{(i)} : \widehat \Omega \to \overline{\Omega^{(i)}},
\quad \mbox{with} \quad \widehat \Omega := [0,1]^2.
\end{equation*}
Here the line above the symbol of a set, like $\overline{\Omega^{(i)}}$, denotes the closure of the respective set. We call the collection $\mathbf G := (\mathbf{G}^{(i)})_{i\in\mathcal{I}_\Omega}$ satisfying these conditions a \emph{multi-patch parameterization} of the domain $\Omega$.

We call the parameterization \emph{$2$-regular} if there is a constant $\underline c$ such that
\begin{equation}\label{eq:lower-bound-Jacobi-det1}
\mathbf G^{(i)} \in W^{2,\infty}(\widehat \Omega)
\quad\mbox{and}\quad
\essinf_{\boldsymbol \xi \in \widehat \Omega}
\det\bigl[\, \nabla {\mathbf{G}^{(i)}}(\boldsymbol \xi) \,\bigr] \geq \underline{c}>0
\quad\mbox{for all}\quad i \in \mathcal I_\Omega,
\end{equation}
where $\nabla {\mathbf{G}^{(i)}}$ denotes the weak Jacobian and $\essinf$ the essential infimum.
Since $\widehat \Omega$ satisfies a strong local Lipschitz condition,
the Sobolev embedding theorem~\citeRefDetail{Theorem~4.12, Part II}{Adams} 
guarantees $\mathbf{G}^{(i)}\in C^1(\widehat\Omega)$.
Since this holds for the \emph{closed} domain $\widehat \Omega$, $\mathbf{G}^{(i)}$ and its derivatives are uniformly bounded.

Next, we discuss the relations of the patches in the overall multi-patch configuration. Let
\begin{equation*}
    \widehat{\mathbf{x}}_1:=(0,0),\;
    \widehat{\mathbf{x}}_2:=(1,0),\;
    \widehat{\mathbf{x}}_3:=(1,1),\;
    \widehat{\mathbf{x}}_4:=(0,1),
\end{equation*}
be the vertices of the parameter domain $\widehat\Omega$. We define the edges of the parameter domain such that each edge $\widehat\Sigma_j$, $j\in  \{1,2,3,4\}$, connects $\widehat{\mathbf{x}}_j$ and $\widehat{\mathbf{x}}_{j+1}$ ($\widehat{\mathbf{x}}_{5}:=\widehat{\mathbf{x}}_{1}$), i.e., having
\begin{equation*}
 \widehat\Sigma_1 := (0,1) \times \{0\},\;
 \widehat\Sigma_2 := \{1\} \times (0,1),\;
 \widehat\Sigma_3 := (0,1) \times \{1\},\;
 \widehat\Sigma_4 := \{0\} \times (0,1).
\end{equation*}
For each edge of the parameter domain $\widehat\Sigma_j$, let $\mathbf n_j:=(n_{j,1},n_{j,2})$ be the outward facing normal vector and $\mathbf t_j := (n_{j,2},-n_{j,1})$ be the tangential vector. The images of the edges and vertices of the parameter domain under $\mathbf G^{(i)}$ are the edges and vertices of the physical patch $\Omega^{(i)}$. 

The parameterization $\mathbf G$ is called \emph{geometrically conforming}, if the closure of $\Omega$ is the disjoint union of all patches, edges and vertices and the intersection of the boundaries of any two patches $\partial\Omega^{(i_1)}\cap \partial \Omega^{(i_2)}$, $i_1\ne i_2 \in \mathcal I_\Omega$, is either a common edge together with the adjacent vertices, a common vertex, or the empty set.

We say that $(i,j)\in\mathcal{I}_{\widehat \Sigma}:= \mathcal{I}_\Omega \times \{1,2,3,4\}$ is an interface, if it is the index of an interior edge $\mathbf{G}^{(i)}(\widehat\Sigma_{j}) \subset \Omega$. Otherwise, it is a boundary edge.
If $(i,j)$ is an interface, there exists exactly one other index $(\neighbori(i,j),\neighborj(i,j)) \in \mathcal{I}_{\widehat \Sigma}$ referring to the same edge, but indexed with respect to the neighboring patch. We say that the multi-patch parameterization \emph{agrees along the interface} if the mapping 
\begin{equation}\label{eq:C0-condition-geometry}
\begin{array}{lccl}
    e_{(i,j)} : &\widehat\Sigma_{j} &\rightarrow&\widehat\Sigma_{\neighborj(i,j)} \\[5pt]
    &\boldsymbol \xi &\mapsto& (\mathbf{G}^{(\neighbori(i,j))})^{-1}\mathbf{G}^{(i)}(\boldsymbol \xi),
\end{array}
\end{equation}
is bijective and affine linear. We then have
\begin{equation*}
    \mathbf{G}^{(i)} |_{\widehat\Sigma_{j}}
    = \mathbf{G}^{(\neighbori(i,j))} \circ e_{(i,j)}.
\end{equation*}

In order to construct an $H^1$-conforming space, we assume that the multi-patch parameterization is analysis-suitable $G^1$, in short AS-$G^1$, as introduced in \citeOneRef{CST16}. For completeness we give the definition below.

\begin{definition}\label{def:ASG1domainAsmpt}
    The parameterization $\mathbf G$ is called \emph{analysis-suitable \(G^1\)} (in short AS-$G^1$), 
    if it is a 
    $2$-regular and 
    geometrically conforming multi-patch parameterization which agrees along all interfaces and if, additionally, for all $(i,j)\in\mathcal{I}_{\widehat \Sigma}$ that are interfaces, i.e., $\mathbf G^{(i)}( \widehat\Sigma_{j})\subset \Omega$, 
    there exist linear functions 
    \begin{equation*}
        \alpha^{(i)}_j,\beta^{(i)}_j
            : [0,1] \to \mathbb{R},
    \end{equation*}
    with
    \begin{equation*}
        \alpha^{(i)}_j(\xi) > 0\quad \mbox{for all}\quad\xi\in[0,1],
    \end{equation*}
    called \emph{gluing data} such that
    \begin{equation}\label{eq:geometry-gluing-condition}
    \left(\mathbf{d}_{j}^{(i)}\cdot\nabla\mathbf{G}^{(i)} \right)|_{\widehat\Sigma_{j}}
    = -\left(\mathbf{d}_{\neighborj(i,j)}^{(\neighbori(i,j))}\cdot\nabla\mathbf{G}^{(\neighbori(i,j))}\right) \circ e_{(i,j)},
    \end{equation}
    holds for the directions 
    \begin{equation}\label{eq:partial d def}
        \mathbf{d}_j^{(i)}(\xi_1,\xi_2) := \frac{1}{\alpha_j^{(i)}(\xi_j)}\left(\mathbf{n}_j + \beta_j^{(i)}(\xi_j) \mathbf{t}_j \right),
    \end{equation}
    where $\xi_3:=\xi_1$ and $\xi_4:=\xi_2$.
    For completeness, for all edges $(i,j)\in \mathcal{I}_{\widehat \Sigma}$ that are boundary edges, i.e., $\mathbf G^{(i)}( \widehat\Sigma_{j})\subset \partial\Omega$, we set 
    \begin{equation*}
        \alpha^{(i)}_j\equiv1 \quad\mbox{and}\quad \beta^{(i)}_j\equiv0
        \quad\mbox{and obtain}\quad
        \mathbf{d}_j^{(i)} = \mathbf{n}_j.
    \end{equation*}
\end{definition}
When convenient, we treat the gluing data $\alpha_j^{(i)}$ and $\beta_j^{(i)}$ as functions defined on $\widehat\Omega$ as obtained using the canonical coordinate projections, like $\alpha_1^{(i)}(\xi_1,\xi_2) := \alpha_1^{(i)}(\xi_1)$.

The concept of AS-$G^1$ parameterizations relies on the geometric continuity of first order, cf. \citeOneRef{peters2002gc}, (in short $G^1$) of the graph surface of the isogeometric function, which is equivalent to the $C^1$-smoothness of the isogeometric function itself, cf. \citeOneRef{groisser2015matched}. 
Note that gluing data, i.e., functions satisfying~\eqref{eq:geometry-gluing-condition}--\eqref{eq:partial d def}, always exists but is not linear in general. Given an AS-$G^1$ parameterization $\mathbf G$, one can compute normalized gluing data as described in \ref{ap:comp-gluing}. 

In order to obtain an upper bound like in~\eqref{eq:hs bound} for some integer $s> 2$, we need additional regularity requirements on the geometry parameterization, beyond condition~\eqref{eq:lower-bound-Jacobi-det1}, as specified in Definitions~\ref{def:s regular as} and~\ref{def:constants}. We can now define isogeometric discretization spaces over analysis-suitable $G^1$ multi-patch domains.

\subsection{Isogeometric function spaces}\label{subsec:isogeometric-functions}

The space of spline functions of degree \( p \in \mathbb{N} \) and smoothness \( C^k \), with \( k < p \), defined over a partition \( Z = (z_0, \ldots, z_n) \) of the interval \([0,1]\), where
\begin{equation*}
0 = z_0 < z_1 < \cdots < z_n = 1,
\end{equation*}
is given by
\begin{equation*}
S_{p,k,Z} := \left\{ v \in C^k(0,1) : \left. v \right|_{[z_{\ell-1}, z_\ell)} \in \mathbb{P}^p \;\mbox{for all}\; \ell = 1, \ldots, n \right\},
\end{equation*}
where \(\mathbb{P}^p\) denotes the space of polynomials of degree at most \(p\), and \(C^k(0,1)\) is the space of all functions that are \(k\)-times continuously differentiable on \([0,1]\). Let
\begin{equation}\label{eq:partition}
    \elemSet_Z := \{(z_{\ell-1}, z_\ell)\}_{\ell=1,\ldots,n}
\end{equation}
denote the set of elements of the partition $Z$. We denote by $h_Z := \max_\ell (z_\ell - z_{\ell-1})$ the grid size, i.e., the size of the largest element in $\elemSet_Z$.

The tensor-product spline space corresponding to the partition \(\mathbf{Z}=(Z_1,Z_2)\) is defined as
\begin{equation*}
S_{p,k,\mathbf{Z}} := S_{p,k,Z_1} \otimes S_{p,k,Z_2}.
\end{equation*}
For convenience, we denote by $Z_3:=Z_1$ and $Z_4:=Z_2$ the partitions corresponding to the edges $\widehat \Sigma_3$ and $\widehat \Sigma_4$, respectively. 
The corresponding grid size is $h_{\mathbf Z}:=\max\{h_{Z_1},h_{Z_2}\}$.
The \(C^1\) isogeometric space is then defined by
\begin{equation}\label{eq:Vh1}
    \mathcal{V}^1_h := \left\{ u_h \in C^1(\overline{\Omega}) :  u_h \circ \mathbf{G}^{(i)} \in S_{p,k,\mathbf{Z}^{(i)}}\;\mbox{for all}\; i \in \mathcal{I}_\Omega \right\},
\end{equation}
where $\mathbf{Z}^{(i)}$ is the partition corresponding to the patch $\Omega^{(i)}$. We denote by $\elemSet_{\mathbf{Z}^{(i)}} := \elemSet_{Z^{(i)}_1} \times \elemSet_{Z^{(i)}_2}$ the set of elements of the mesh and define the mesh size corresponding to the isogeometric space $\mathcal{V}^1_h$ in the physical domain as
\begin{equation}\label{eq:phys h}
    \physH := \max_{i\in \mathcal{I}_\Omega} \max_{E\in \elemSet_{\mathbf{Z}^{(i)}}} \mbox{diam} \left(\mathbf{G}^{(i)}\left(E\right)\right),
\end{equation}
where $\mbox{diam} (\cdot)$ denotes the diameter of the element.

By definition, the isogeometric space satisfies \(\mathcal{V}^1_h \subset C^0(\overline{\Omega})\), which, for a function \(u_h \in \mathcal{V}^1_h\) with pre-images $\widehat u^{(i)}_h := u\circ \mathbf G^{(i)}$ and every $j\in  \{1,2,3,4\}$, such that \(\mathbf{G}^{(i)}(\widehat\Sigma_j)\) is an  interface, results in the condition
\begin{equation*}
    \widehat u^{(i)}_h |_{\widehat\Sigma_j} =\widehat u^{(\neighbori(i,j))}_h \circ e_{(i,j)},
\end{equation*}
where $e_{(i,j)}$ is the isometry defined in~\eqref{eq:C0-condition-geometry}. We assume that the spline spaces  \emph{match along all interfaces}, i.e.,
for all $(i,j)\in\mathcal{I}_{\widehat\Sigma}$ that are interfaces, we have
\begin{equation*}
    S_{p,k,\mathbf{Z}^{(i)}}\big|_{\widehat\Sigma}
    =
    S_{p,k,\mathbf{Z}^{(\neighbori(i,j))}}\circ e_{(i,j)}.
\end{equation*}
Since the assumption is symmetric, the converse is also true. This is the case if and only if the partitions along the interface match with
\begin{equation}\label{eq:edge-Z-def}
    Z^{(\neighbori(i,j))}_{\neighborj(i,j)} \in \left\{ Z^{(i)}_j ,\overline Z^{(i)}_j \right\},
\end{equation}
where $\overline Z = (1-z_n,\ldots,1-z_0)$ is the reverse of the partition $Z=(z_0,\ldots,z_n)$. The correct choice is determined by the direction of the embedding $e_{(i,j)}$, 
see Fig.~\ref{fig:geometric} for an example configuration. It is possible to assume only partially matching spaces, i.e., for each pair $(i,j)$ and $(\neighbori(i,j),\neighborj(i,j))$ one of the spaces is a proper subspace of the other. All results carry over with only minor modifications. To not further complicate the notation, we assume fully matching spaces at all interfaces.

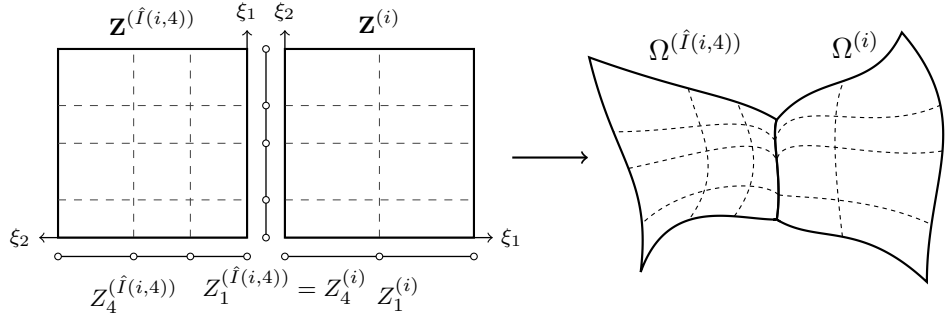
\begin{figure}[h!]
    \centering
    \begin{tikzpicture}[scale=2.5]
        \tikzset{
            patch border/.style={draw, black, line width=0.8pt},
            knot line/.style={draw, darkgray, line width=0.3pt, dashed},
            axis line/.style={draw, black, ->, line width=0.5pt},
            knot circle/.style={fill=white, draw=black, line width=0.3pt, radius=0.5pt}
        }
        \begin{scope}[local bounding box=paramDomain]
            \def\vKnots{0.2, 0.5, 0.7}   
            \def\uKnotsLeft{0.7, 0.4}
            \def\uKnotsRight{0.5}
            \draw[patch border] (-1.1, 0) rectangle (-0.1, 1);
            \foreach \y in \vKnots {
                \draw[knot line] (-1.1, \y) -- (-0.1, \y);
            }
            \foreach \xShift in \uKnotsLeft {
                \draw[knot line] (-1.1 + \xShift, 0) -- (-1.1 + \xShift, 1);
            }
            \draw[patch border] (0.1, 0) rectangle (1.1, 1);
            \foreach \y in \vKnots {
                \draw[knot line] (0.1, \y) -- (1.1, \y);
            }
            \foreach \xShift in \uKnotsRight {
                \draw[knot line] (0.1 + \xShift, 0) -- (0.1 + \xShift, 1);
            }
            \draw[line width=0.5pt] (0, 0) -- (0, 1.0);
            \foreach \y in {0, 1} \draw[fill=white] (0,\y) circle [knot circle]; 
            \foreach \y in \vKnots \draw[fill=white] (0,\y) circle [knot circle]; 
            \draw[line width=0.5pt] (-0.1, -0.1) -- (-1.1, -0.1);
            \foreach \xShift in {0, 1} \draw[fill=white] (-1.1 + \xShift, -0.1) circle [knot circle];
            \foreach \xShift in \uKnotsLeft \draw[fill=white] (-1.1 + \xShift, -0.1) circle [knot circle];

            \draw[line width=0.5pt] (0.1, -0.1) -- (1.1, -0.1); 
            \foreach \xShift in {0, 1} \draw[fill=white] (0.1 + \xShift, -0.1) circle [knot circle]; 
            \foreach \xShift in \uKnotsRight \draw[fill=white] (0.1 + \xShift, -0.1) circle [knot circle]; 

            \draw[axis line] (-0.1, 0) -- (-0.1, 1.1);
            \draw[axis line] (0.1, 0) -- (0.1, 1.1);

            \draw[axis line] (-0.1, 0) -- (-1.2, 0);
            \draw[axis line] (0.1, 0) -- (1.2, 0);
            
            \node[align=center, font=\footnotesize] at (0.1, 1.2) {$\xi_2$};
            \node[align=center, font=\footnotesize] at (-0.1, 1.2) {$\xi_1$};

            \node[align=center, font=\footnotesize] at (1.3, 0.0) {$\xi_1$};
            \node[align=center, font=\footnotesize] at (-1.3, 0.0) {$\xi_2$};

            \node[font=\small] at (0.1, -0.25) {$Z^{(\neighbori(i,4))}_1=Z^{(i)}_4$};
            \node[font=\small] at (-0.7, -0.3) {$Z^{(\neighbori(i,4))}_4$};
            \node[font=\small] at (0.7, -0.3) {$Z^{(i)}_1$};
            
            \node[font=\small] at (-0.6, 1.15) {$\mathbf{Z}^{(\neighbori(i,4))}$};
            \node[font=\small] at (0.6, 1.15) {$\mathbf{Z}^{(i)}$};

        \end{scope}
       \node[inner sep=0pt, anchor=west] (pdf) at ([xshift=0.2cm] paramDomain.east)
        {\includegraphics[width=0.41\textwidth]{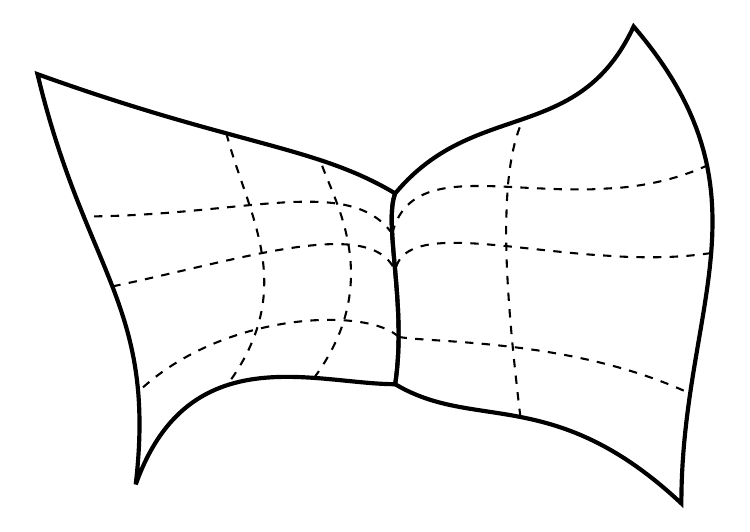}};
        \node[black, font=\large] at ($(pdf.west)!0.25!(pdf.east) + (0.15, 0.6)$) {$\Omega^{(\neighbori(i,4))}$};
        \node[black, font=\large] at ($(pdf.west)!0.75!(pdf.east) + (-0.05, 0.6)$) {$\Omega^{(i)}$};

        \draw[->, thick, black] ([xshift=-0.1cm]paramDomain.east) -- ([xshift=0.1cm]pdf.west);
    \end{tikzpicture}
    
    \caption{The parametric domain $\widehat{\Omega}^{(i)}$ and $\widehat{\Omega}^{(\hat{I}(i,4))}$ (left) are mapped via $\mathbf{G}^{(i)}$ and $\mathbf{G}^{(\hat{I}(i,4))}$ to the physical patches ${\Omega}^{(i)}$ and ${\Omega}^{(\hat{I}(i,4))}$ (right), respectively. Partitions $\mathbf{Z}^{(i)}=(Z_1^{(i)},Z_2^{(i)})$ and $\mathbf{Z}^{(\hat{I}(i,4))}=(Z_1^{(\hat{I}(i,4))},Z_2^{(\hat{I}(i,4))})$ match along the interface with $Z_2^{(i)}=Z_4^{(i)}=Z_1^{(\hat{I}(i,4))}$. The dashed lines represent the inner knots and the parameter directions are denoted by $\xi_1$ and $\xi_2$.}
    \label{fig:geometric}
\end{figure}

We define crossing directions $\mathfrak{d}^{(i)}_j:\Omega^{(i)} \to\mathbb{R}^2$ for all $(i,j)$ as
\begin{equation}\label{eq:directions}
    \mathfrak{d}^{(i)}_j\circ \mathbf{G}^{(i)} := \mathbf{d}_{j}^{(i)}\cdot\nabla\mathbf{G}^{(i)}.
\end{equation}
Consider, without loss of generality, the edge $\widehat \Sigma_1$. 
Following~\citeRefDetail{eq.~(29)}{CST16}, the Taylor expansion around $\xi_2=0$ is
\begin{equation*}
    u_h \circ \mathbf{G}^{(i)} (\xi_1,\xi_2) = g_0(\xi_1) + (\beta_1^{(i)}(\xi_1)g_0'(\xi_1)+\alpha_1^{(i)}(\xi_1)g_1(\xi_1))\xi_2 + O(\xi_2^2),
\end{equation*}
where 
\begin{equation*}
    g_0(\xi_1) := u_h \circ \mathbf{G}^{(i)}(\xi_1,0) \quad \text{and}\quad g_1(\xi_1):=\left( \mathfrak{d}^{(i)}_j\cdot\nabla u_h \right) \circ \mathbf{G}^{(i)}(\xi_1,0).
\end{equation*}
One can see that $g_0$ and $g_1$ have to be splines, cf.~\citeRefDetail{Theorem 1}{CST16}. Since $\beta_1^{(i)}g_0'+\alpha_1^{(i)}g_1 \in S_{p,k,{Z}^{(i)}_1}$, for generic linear gluing data, the maximal spaces that can be chosen independently are $S_{p,k+1,{Z}^{(i)}_1}$ for the trace $g_0$ and $S_{p-1,k,{Z}^{(i)}_1}$ for the directional derivative $g_1$. These spaces reduce to polynomials, if and only if $k\ge p-1$, which leads to $C^1$ locking, cf.~\citeRefDetail{Theorem 2}{CST16}. Thus, we assume $1\leq k \leq p-2$. Based on these observations, we define the AS-$G^1$ subspace of $\mathcal{V}^{1}_h$ as 
\begin{equation}\label{eq:V1 AS}
    \mathcal{V}^{1}_{AS} 
    :=  \left\{ u_h \in \mathcal{V}^1_h : 
    \begin{array}{ccll}
    u_h \circ \mathbf{G}^{(i)}|_{\widehat\Sigma_j} &\in& S_{p,k+1,{Z}^{(i)}_j} \\
    (\mathfrak{d}^{(i)}_j\cdot\nabla u_h ) \circ \mathbf{G}^{(i)}|_{\widehat\Sigma_j} &\in& S_{p-1,k,{Z}^{(i)}_j}
    \end{array}
    \;\mbox{for all}\; (i,j)\in \mathcal{I}_{\widehat\Sigma}
    \right\}.
\end{equation}
Following Definition~\ref{def:ASG1domainAsmpt}, we then have for all interfaces
\begin{equation*}
    \mathfrak{d}^{(i)}_j (\mathbf x) = -\mathfrak{d}^{(\neighbori(i,j))}_{\neighborj(i,j)} (\mathbf x) \quad\mbox{for all}\quad \mathbf{x} \in \mathbf{G}^{(i)}(\widehat\Sigma_j).
\end{equation*}
We moreover define the Argyris-like subspace as
\begin{equation}\label{eq:V1star}
    \mathcal{V}^{1}_{\mathcal{A}} 
    := \left\{ u_h \in \mathcal{V}^1_{AS} : \;
    u_h \in C^2(\mathbf{x}) \;\mbox{for all}\;\mathbf{x}=\mathbf{G}^{(i)}(\widehat{\mathbf{x}}_j)\;\mbox{with}\; (i,j)\in \mathcal{I}_{\widehat{\mathbf x}} 
    \right\},
\end{equation}
where $\mathcal{I}_{\widehat{\mathbf x}} := \mathcal{I}_\Omega \times\{1,2,3,4\}$ is the index set of all vertices. It employs $C^2$ super-smoothness at all vertices, similar to the Argyris finite element, cf. \citeOneRef{Argyris1968TUBA}, to obtain a space whose dimension is independent of the geometry.

\begin{remark}
Note that $\mathcal{V}^{1}_{AS} \cap H^2_{0} =\mathcal{V}^{1}_h \cap H^2_{0}$, if and only if, for all $(i,j)\in \mathcal{I}_{\widehat\Sigma}$ that are interfaces, with $(\neighbori,\neighborj)=(\neighbori(i,j),\neighborj(i,j))$, we have 
\begin{equation*}
    \max \left\{\mathrm{deg}(\alpha^{(i)}_{j}), \mathrm{deg}(\alpha^{(\neighbori)}_{\neighborj})\right\}=1
\end{equation*}
and no zero of $\alpha^{(i)}_{j}(\beta^{(\neighbori)}_{\neighborj}\circ r)+\beta^{(i)}_{j}(\alpha^{(\neighbori)}_{\neighborj}\circ r)$ is a break point of $Z^{(i)}_j$, where $r(\xi)=1-\xi$, if $e_{(i,j)}$ flips the direction, and $r(\xi)=\xi$, otherwise. This follows from Lemmas~3 and~4 in \citeOneRef{Kapl2017}.
The spaces without $H^2_0$ boundary conditions satisfy the equality, if moreover for all boundary edges we assume that the traces and directional derivatives in~\eqref{eq:V1 AS} are in the full spline spaces $S_{p,k,{Z}^{(i)}_j}$. We avoid this distinction to simplify the presentation.

It is possible to describe the space $\mathcal{V}^{1}_h$ exactly, by modifying the conditions that appear in~\eqref{eq:V1 AS} according to the polynomial degrees of the gluing data for each interface. However, to prove approximation properties of $\mathcal{V}^{1}_h$ we can restrict ourselves to the subspace $\mathcal{V}^{1}_{AS}$.
\end{remark}
The structure of the AS-$G^1$ and Argyris-like subspaces allows a patch-local description, which is summarized in the following. For every patch $\Omega^{(i)}$, $i\in \mathcal{I}_\Omega$, let
\begin{equation*}
    \mathcal{S}_{AS}^{(i)} := \left\{ u_h \circ \mathbf{G}^{(i)} : u_h \in \mathcal{V}^{1}_{AS}\right\} \subseteq S_{p,k,\mathbf{Z}^{(i)}}
\end{equation*}
be the patch-local AS-$G^1$ subspace which can be described in the following way.
\begin{lemma}\label{lem:continuity_new_notaion}
    Let $\Omega^{(i)}$, $i\in \mathcal I_\Omega$, be a fixed patch.
    The space $\mathcal{S}_{AS}^{(i)}$ is the space of all $\widehat{u}^{(i)} \in S_{p,k,\mathbf{Z}^{(i)}}$, where, for all $j\in  \{1,2,3,4\}$,
    \begin{equation}\label{eq:pullback-traces}
        \widehat u^{(i)}_h|_{\widehat \Sigma_j} \in S_{p,k+1,Z^{(i)}_j}
    \end{equation}
    and
    \begin{equation}\label{eq:pullback-d-derivatives}
        \left(\mathbf{d}_j^{(i)} \cdot \nabla \widehat u^{(i)}_h\right)|_{\widehat \Sigma_j} \in S_{p-1,k,Z^{(i)}_j}.
    \end{equation}
\end{lemma}
\begin{proof}
Identity \eqref{eq:pullback-traces} is a direct consequence of~\eqref{eq:V1 AS}. For the identity~\eqref{eq:pullback-d-derivatives} involving the directional derivative, applying the chain rule to $\widehat{u}_h^{(i)} = u_h \circ \mathbf{G}^{(i)}$ and using~\eqref{eq:directions} yields
\begin{equation*}
    \mathbf{d}_j^{(i)} \cdot \nabla \widehat{u}_h^{(i)} = \mathbf{d}_j^{(i)} \cdot \nabla \mathbf{G}^{(i)} (\nabla u_h \circ \mathbf{G}^{(i)}) = \left(\mathfrak{d}^{(i)}_j \cdot \nabla u_h \right) \circ \mathbf{G}^{(i)},
\end{equation*}
which concludes the proof.
\end{proof}
For any $C^1$-smooth isogeometric function $u_h\in \mathcal{V}^{1}_{AS}$ and any interface between two neighboring patches, the corresponding traces~\eqref{eq:pullback-traces} as well as the directional derivatives~\eqref{eq:pullback-d-derivatives} coincide. In the following section, we prove approximation error estimates for this space, using a patch-wise quasi-interpolator that is defined by projecting traces and directional derivatives into the corresponding spline spaces, following Lemma~\ref{lem:continuity_new_notaion}.

\section{Approximation error estimates}\label{sec:approx-error-estimates}

In this section, we prove the main theorem, which provides the approximation error estimates and interpolation properties of the global AS-$G^1$ projector. First, we recall error estimates for univariate Ritz-type projectors, which we extend to tensor-product spaces. We then define a patch-wise AS-$G^1$ projector with suitable interpolation properties, which results in matching traces and gradients at all interfaces. At last, we provide the proof of the main theorem and show for which geometries it can be applied. 

\subsection{The main theorem}\label{subsec:main-theorem}

In order to provide the main quasi interpolation error estimate, we need that the function $u\in H^2(\Omega)$, which is to be approximated, is sufficiently smooth within each patch. This is described for all $s\ge 2$ by the bent Sobolev space
\begin{equation*}
    \mathcal{H}^s(\Omega) := \left\{ u \in H^2(\Omega) : u|_{\Omega^{(i)}} \in H^s (\Omega^{(i)})\mbox{ for all } i \in \mathcal{I}_{\Omega}\right\},
\end{equation*}
with norm
\begin{equation*}
    \|u\|_{\mathcal{H}^s(\Omega)}
    :=
    \left(
    \sum_{i \in \mathcal{I}_{\Omega}}
    \|u\|_{H^s(\Omega^{(i)})}^2
    \right)^{1/2}.
\end{equation*}
Certainly, $\|u\|_{\mathcal{H}^s(\Omega)}=\|u\|_{H^s(\Omega)}$ for all $u\in H^s(\Omega)$.

Moreover, on $(0,1)$ and on the parameter domain $\widehat\Omega$, we define the broken Sobolev spaces over the partitions $\elemSet_Z$, see~\eqref{eq:partition}, and $\mathbf{Z}^{(i)}$ via
\begin{equation*}
\begin{aligned}
    \mathcal{H}^s(Z;(0,1)) &:= \left\{ u \in L^2(0,1) : u|_{E} \in H^s (E)\mbox{ for all } E \in \elemSet_{Z}\right\},\\
    \mathcal{H}^s(\mathbf{Z}^{(i)};\widehat\Omega) &:= \left\{ u \in L^2(\widehat\Omega) : u|_{E} \in H^s (E)\mbox{ for all } E \in \elemSet_{\mathbf{Z}^{(i)}}\right\},
\end{aligned}
\end{equation*}
equipped with the broken semi-norms
\begin{equation*}
    |u|_{\mathcal{H}^s(Z;(0,1))}
    :=
    \left(
    \sum_{E \in \elemSet_{Z}}
    |u|_{H^s(E)}^2
    \right)^{1/2}
\quad
\mbox{and}
\quad
    |u|_{\mathcal{H}^s(\mathbf{Z}^{(i)};\widehat\Omega)}
    :=
    \left(
    \sum_{E \in \elemSet_{\mathbf{Z}^{(i)}}}
    |u|_{H^s(E)}^2
    \right)^{1/2},
\end{equation*}
respectively.

For higher Sobolev regularities we have to require additional smoothness conditions on the parameterization, which extend the concept of $2$-regularity defined in~\eqref{eq:lower-bound-Jacobi-det1}.
\begin{definition}\label{def:s regular as}
    We call the parameterization $s$-regular 
    for any integer $s > 2$ if
    besides \eqref{eq:lower-bound-Jacobi-det1} additionally 
    \begin{equation*}
        \mathbf{G}^{(i)} |_{\widehat E}
        \in
        W^{s,\infty}(\widehat E),
        \quad \mbox{for all}\quad i\in\mathcal{I}_\Omega\mbox{ and for all }\widehat E\in \elemSet_{\mathbf{Z}^{(i)}}
        .
    \end{equation*}
\end{definition}
Isogeometric parameterizations, i.e., $\mathbf{G}^{(i)} \in (S_{p,k,\mathbf{Z}^{(i)}})^2$, with $k\geq 1$, are $\infty$-regular.
\begin{definition}\label{def:constants}
    We use the notation $C(\mathbf{G},s)$ to denote a generic constant that depends on the value of the smoothness parameter $s\ge 2$ and
    \begin{equation*}
    \begin{aligned}   
        &\max_{i \in \mathcal I_\Omega} \| \mathbf{G}^{(i)} \|_{W^{s,\infty}(\widehat\Omega)} ,
        \quad
        \max_{i \in \mathcal I_\Omega} \| (\det\nabla\mathbf{G}^{(i)})^{-1} \|_{L^{\infty}(\widehat\Omega)},\\
        \quad
        &\max_{(i,j)\in \mathcal I_{\widehat \Sigma}} \| \alpha_j^{(i)} \|_{L^{\infty}(0,1)},
        \quad
        \max_{(i,j)\in \mathcal I_{\widehat \Sigma}} \| \beta_j^{(i)} \|_{L^{\infty}(0,1)}.
        \end{aligned}
    \end{equation*}
    We use the notation $C(s)$ and $C(p)$ to denote generic constants depending only on those variables. 
    We use $C$ to denote a universal constant, i.e., one that is independent of any parameters. Any of these generic constants can have a different value in any instance where it is used.
\end{definition}
Since one can compute normalized gluing data as described in \ref{ap:comp-gluing}, we think of the gluing data as a quantity that depends on the parameterization $\mathbf G$.
The dependence on the parameterization can be quantified in case of bilinear patches or in case of AS-$G^1$ reparameterizations, following the construction from \citeOneRef{kapl2018construction}.

Using these assumptions, we are now in the position to state the main quasi interpolation error estimate.
\begin{theorem}\label{thrm:main}
    Let $p$, $k$, $t$ and $s$ be integers with $0\leq t \leq 2 < k+2 \leq p$ and $4 \leq s \leq p+1$. Let $\mathcal V_{AS}^1$ be the AS-$G^1$ space over an $s$-regular analysis-suitable $G^1$ parameterization, where $h_Z\leq 1/(p+1)$ for all underlying partitions $Z$.
    There is a projector
    $\tpprojectorSph: \mathcal H^4(\Omega) \to \mathcal{V}^1_{AS}$ and a constant $C(\mathbf{G},s)$ that only depends on the parameterization and Sobolev order, such that 
    \begin{equation}\label{eq:thrm:main:estimate}
        \|u-\tpprojectorSph u\|_{H^t(\Omega)}
        \leq C(\mathbf{G},s) \physH^{s-t} \|u\|_{\mathcal{H}^s(\Omega)}
        \quad\mbox{for all}\quad u\in \mathcal{H}^s(\Omega).
    \end{equation}
    Moreover, we have for all $(i,j)\in\mathcal{I}_\Omega\times\{1,2,3,4\}$
    \begin{align}
        \label{eq:thrm:main:cond-1}
        u\in C^2(\mathbf x)
        &\quad \Rightarrow \quad
        \tpprojectorSph u\in C^2(\mathbf x)
        && \mbox{for any} \quad \mathbf x = \mathbf{G}^{(i)}(\widehat{\mathbf{x}}_j),
        \\
        \label{eq:thrm:main:cond-2}
        u\big|_{\Sigma} = 0
        &\quad \Rightarrow \quad
        (\tpprojectorSph u)\big|_{\Sigma} = 0
        && \mbox{for any} \quad \Sigma= \mathbf{G}^{(i)}(\widehat{\Sigma}_j),
        \\
        \label{eq:thrm:main:cond-3}
        \nabla u\big|_{\Sigma} = \mathbf 0
        &\quad \Rightarrow \quad
        \nabla(\tpprojectorSph u)\big|_{\Sigma} = \mathbf 0
        && \mbox{for any} \quad \Sigma= \mathbf{G}^{(i)}(\widehat{\Sigma}_j).
    \end{align}
\end{theorem}
We give the proof of this theorem at the end of this section.

Equation~\eqref{eq:thrm:main:cond-1} guarantees that $\Pi_{AS}$ projects into $\mathcal V^1_{\mathcal A}$ if the function $u$ is $C^2$-smooth at all vertices $\mathbf x^{(l)}$. According to the Sobolev embedding theorem, \citeRefDetail{Theorem~4.12, Part II}{Adams} this is the case if $u\in H^4(\Omega)$.
The conditions~\eqref{eq:thrm:main:cond-2} and~\eqref{eq:thrm:main:cond-3} are relevant if essential boundary conditions are present: If $u$ vanishes on some (boundary) edge $\Sigma$, then~\eqref{eq:thrm:main:cond-2} states that $\Pi_{AS} u$ also vanishes there. If both $u$ and $\partial_n u$ vanish on some (boundary) edge $\Sigma$, then the combination of~\eqref{eq:thrm:main:cond-2} and~\eqref{eq:thrm:main:cond-3} guarantees that also $\Pi_{AS} u$ and $\partial_n \Pi_{AS} u$ vanish on that edge.

If $s<4$ or if $s$ is not an integer, we can still provide an approximation error estimate.

\begin{corollary} 
    Let $p$ and $k$ be integers and let $s\in\mathbb{R}$ with $2 < k+2 \leq p$ and $2 \leq s \leq p+1$. Let $s':=\max\{\lceil s\rceil,4\}$ and $\mathcal V_{\mathcal A}^1$ be the Argyris-like space over a $s'$-regular analysis-suitable $G^1$ parameterization, where $h_Z\le 1/(p+1)$ for all underlying partitions $Z$.
    There is a constant $C(\mathbf{G},s')$ that only depends on the parameterization and the Sobolev order such that 
    \begin{equation*}
        \inf_{u_h\in \mathcal V_{\mathcal A}^1}
        \|u-u_h\|_{H^2(\Omega)}
        \leq C(\mathbf{G},s') \physH^{s-2} \|u\|_{H^s(\Omega)}
        \quad\mbox{for all}\quad u\in H^s(\Omega).
    \end{equation*}
\end{corollary}
\begin{proof}
    Using Theorem~\ref{thrm:main}, particularly~\eqref{eq:thrm:main:estimate} and~\eqref{eq:thrm:main:cond-1}, in combination with~\eqref{eq:V1star}, we obtain for all integers $\sigma$ with $4 \leq \sigma\leq p+1$
    \begin{equation}\label{eq:corollary:main-1}
        \inf_{u_h\in \mathcal V_{\mathcal A}^1} \|u-u_h\|_{H^2(\Omega)}
        \le C(\mathbf{G},\sigma) \physH^{\sigma-2} \|u\|_{H^{\sigma}(\Omega)}
        \quad\mbox{for all}\quad u\in H^{\sigma}(\Omega) .
    \end{equation}
    This gives the desired result for integer $s$ with $4\leq s \leq p+1$.  
    For non-integer $s$ with $4 <s< p+1$, the result is obtained by Hilbert space interpolation using $\sigma=\lfloor s \rfloor$ and $\sigma=\lceil s \rceil$ (cf. the interpolation theorem \citeRefDetail{Theorem~7.23}{Adams} and the equivalence of Sobolev spaces with the corresponding interpolation spaces \citeRefDetail{Chapter~1, Theorem~9.6}{Lions1972}).
    Since $0\in\mathcal V_{\mathcal A}^1$, we also have
    \begin{equation}\label{eq:corollary:main-2}
        \inf_{u_h\in \mathcal V_{\mathcal A}^1} \|u-u_h\|_{H^2(\Omega)}
        \le \|u\|_{H^2(\Omega)}
        \quad\mbox{for all}\quad u\in H^2(\Omega) ,
    \end{equation}
    which shows the desired result for $s=2$.
    The result for $2<s<4$ also follows using Hilbert space interpolation, specifically, if one interpolates between~\eqref{eq:corollary:main-1} with $\sigma=4$ and~\eqref{eq:corollary:main-2}.
    This finishes the proof.
\end{proof}

In the following, we discuss to which classes of geometry parameterizations the main theorem can be applied. While the presented function spaces are naturally motivated by the isogeometric framework, our theoretical results do not strictly require an isoparametric approach. 
Specifically, Theorem~\ref{thrm:main} holds for any multi-patch parameterization that satisfies the $s$-regular analysis-suitable $G^1$ conditions established in Definitions \ref{def:ASG1domainAsmpt}, \ref{def:s regular as}, and \ref{def:constants}. 
Consequently, the derived approximation error estimates remain valid for a range of configurations, including standard AS-$G^1$ B-spline discretizations, non-isoparametric AS-$G^1$ NURBS discretizations, and $C^1$ finite element spaces over geometrically conforming bilinear quadrilateral meshes, which can be interpreted as multi-patch B-splines. In the latter case, our theory applies after sufficiently many refinements, such that $h\leq 1/(p+1)$, and generalizes the results from \citeOneRef{Kapl2021}. 
For standard AS-$G^1$ B-spline geometries, the isogeometric space is defined via the pullback
\begin{equation*}
    \mathcal{V}_h^1 := \left\{ u_h \in C^1(\Omega) : u_h \circ \mathbf{G}^{(i)} \in S_{p,k,\mathbf{Z}^{(i)}} \text{ where }\mathbf{G}^{(i)} \in \left(S_{p,k,\mathbf{Z}^{(i)}}\right)^2\text{ for all } i \in \mathcal{I}_\Omega \right\}.
\end{equation*}
Geometric continuity requires matching traces and coplanar directional derivatives across all interfaces, enforcing the AS-$G^1$ condition
\begin{equation}\label{eq:interface continuity}
 \begin{aligned}
     \mathbf{G}^{(i)}|_{\widehat\Sigma_{j}} &= \mathbf{G}^{(\neighbori(i,j))}\circ e_{(i,j)},\\
     \left(\mathbf{d}_{j}^{(i)}\cdot\nabla\mathbf{G}^{(i)} \right)|_{\widehat\Sigma_{j}}
    &= -\left(\mathbf{d}_{\neighborj(i,j)}^{(\neighbori(i,j))}\cdot\nabla\mathbf{G}^{(\neighbori(i,j))}\right) \circ e_{(i,j)}.
 \end{aligned}
\end{equation}
For AS-$G^1$ NURBS geometries 
\begin{equation*}
    \mathbf{G}^{(i)} = \frac{\mathbf{F}^{(i)}}{w^{(i)}} \quad\mbox{with}\quad w^{(i)}\in S_{p,k,\mathbf{Z}^{(i)}}\quad\mbox{and}\quad \mathbf{F}^{(i)}\in \left(S_{p,k,\mathbf{Z}^{(i)}}\right)^2\quad\mbox{for all}\quad i \in \mathcal{I}_\Omega,
\end{equation*}
utilizing non-isogeometric discretizations, where the functions are polynomial splines on the parameter domain
\begin{equation*}
    \mathcal{V}_h^1 := \left\{ u_h \in C^1(\Omega) : u_h \circ \mathbf{G}^{(i)} \in S_{p,k,\mathbf{Z}^{(i)}} \text{ for all } i \in \mathcal{I}_\Omega \right\},
\end{equation*}
the rational multi-patch parameterization $\mathbf{G}^{(i)}$ needs to satisfy the geometric continuity constraints~\eqref{eq:interface continuity} directly. The $G^1$-continuity of the physical geometry parameterization $\mathbf{G}^{(i)} = \mathbf{F}^{(i)} / w^{(i)}$ does not necessarily require the homogeneous components $\mathbf{F}^{(i)}$ and $w^{(i)}$ to independently satisfy these interface matching conditions, but if they do, then the AS-$G^1$ conditions also hold for $\mathbf{G}^{(i)}$.

In the following we derive the estimates for univariate and tensor-product splines that we need in order to prove the main theorem.

\subsection{Estimates for univariate splines}\label{subsec:univ-estim}

For our analysis, we use the existence of Ritz-type projectors that satisfy certain Hermite interpolation properties and approximation error estimates, as developed in \citeOneRef{sande2022ritz}.

\begin{lemma}\label{lem:1D_edge_cond}
    Let $p$, $k$ and $r$ be integers with
    $0\le r\le k+1 \le p$. 
    There is a projector $\Pi_{p,k,Z}^{(r)}:H^r(0,1) \to S_{p,k,Z}$
    such that the following statements hold.
    \\
    (a) $\Pi_{p,k,Z}^{(r)}$ is $H^r$-orthogonal to $S_{p,k,Z}$, i.e.,
    \begin{equation*}
    \left(\partial^{r}(u-\Pi_{p,k,Z}^{(r)} u),\partial^{r}v_h\right)_{L^2(0,1)}  = 0
    \quad
    \mbox{for all}
    \;
    v_h\in S_{p,k,Z}
    \;
    \mbox{and}
    \;
    u\in H^r(0,1).
    \end{equation*}
    (b) $\Pi_{p,k,Z}^{(r)}$ interpolates on the left side, i.e.,
    \begin{equation*}
    \partial^{s}\Pi_{p,k,Z}^{(r)} u(0)=\partial^{s} u(0)
    \quad
    \mbox{for all integers}\;
    s
    \;\mbox{with} \;
    0 \leq s\leq r-1
    \;\mbox{and}\;
    u\in H^r(0,1).
    \end{equation*}
    (c) If additionally
    $p \ge 2r-1$, 
    it also interpolates on the right side, i.e.,
    \begin{equation*}
    \partial^{s}\Pi_{p,k,Z}^{(r)} u(1)=\partial^{s} u(1)
    \quad
    \mbox{for all integers}\;
    s
    \;\mbox{with} \;
    0 \leq s\leq r-1
    \;\mbox{and}\;
    u\in H^r(0,1).
    \end{equation*}
    (d) For $t$ and $s$ integers with $0\le t\le r\le s \le p+1$ and $t\ge 2r-p-1$ and all
    $ u\in  H^{\min\{s,k+1\}}(0,1)\cap\mathcal{H}^s(Z;(0,1))$,
    the following error estimate holds:
    \begin{equation}\label{eq:lem:1Drst}
        |u-\Pi_{p,k,Z}^{(r)} u|_{H^t(0,1)} \le  \left(\frac{h_Z}{\pi}\right)^{s-t} |u|_{\mathcal{H}^s(Z;(0,1))}.
    \end{equation}
\end{lemma}
\begin{proof}
    We use the projector from~\citeRefDetail{Eq.~(6)}{sande2022ritz}.
    For that projector, statement (a) is given as~\citeRefDetail{Eq.~(7)}{sande2022ritz},
    statement (b) is given as~\citeRefDetail{Eq.~(8)}{sande2022ritz},
    statement (c) is given as~\citeRefDetail{Lemma~2}{sande2022ritz}. 
    Moreover, \citeRefDetail{Theorem~3}{sande2022ritz} states
    \begin{equation}\label{eq:lem:1Drst-nonbroken}
        |u-\Pi_{p,k,Z}^{(r)} u|_{H^t(0,1)} \le  \left(\frac{h_Z}{\pi}\right)^{s-t} |u|_{{H}^s(0,1)}
        \quad \mbox{for all}\quad u\in  H^s(0,1),
    \end{equation}
    which finishes the proof for $s\le k+1$. 
    To handle the case $s>k+1$, we introduce correction terms similar to~\citeRefDetail{Lemma 6}{bressan2019approximation}. Specifically, there exists a spline $g \in S_{s-1,k,Z}\subseteq S_{p,k,Z}$, such that $w\coloneq u-g \in H^s(0,1)$.
    Thus, we can apply~\eqref{eq:lem:1Drst-nonbroken} and obtain
    \begin{equation} \label{eq:global_w_bound}
    \begin{aligned}
        |u - \Pi_{p,k,Z}^{(r)} u|_{H^t(0,1)} &=|(w + g) - \Pi_{p,k,Z}^{(r)} (w + g)|_{H^t(0,1)}
        \\ &=|w - \Pi_{p,k,Z}^{(r)} w|_{H^t(0,1)} \le \left(\frac{h_Z}{\pi}\right)^{s-t} |w|_{H^s(0,1)}.
        \end{aligned}
    \end{equation}
     On any element $E \in \mathcal{E}_Z$, the function $g$ is a polynomial of degree at most $s-1$. Consequently, $\partial^s g|_E \equiv 0$. This yields 
    \begin{equation*}
        |w|_{H^s(0,1)}^2 = \sum_{E \in \mathcal{E}_Z} \|\partial^s w\|_{L^2(E)}^2 = \sum_{E \in \mathcal{E}_Z} \|\partial^s u\|_{L^2(E)}^2 = |u|_{H^s(Z;(0,1))}^2.
    \end{equation*}
    Substituting this norm equivalence into \eqref{eq:global_w_bound} yields the desired upper bound \eqref{eq:lem:1Drst}, which completes the proof.
    \end{proof}
The constant $1/\pi^{s-t}$ in~\eqref{eq:lem:1Drst} is that for splines of maximum smoothness. For $k\le p-2$, there exist sharper but less elegant constants, cf. \citeRefs{S19,sande2022ritz}. 
The following lemma establishes a scaled one-dimensional trace inequality for the boundary value of a function by its $L^2$ and $H^1$ norms.

\begin{lemma}\label{lem:1d trace}
    For all $u\in H^1(0,1)$ and all $\eta \in (0,1]$, we have
    \begin{equation*}
        |u(0)| \le \sqrt{2\eta^{-1} \|u\|_{L^2(0,1)}^2 + \eta |u|_{H^1(0,1)}^2}
        \le \sqrt{2}\eta^{-1/2} \|u\|_{L^2(0,1)} + \eta^{1/2} |u|_{H^1(0,1)}.
    \end{equation*}
\end{lemma}
\begin{proof}
    Applying the fundamental theorem of calculus, the Cauchy--Schwarz and Young's inequalities yields the desired result.
\end{proof}

Next, we establish the existence of localized spline functions that interpolate boundary derivatives up to order two and satisfy corresponding bounds on their Sobolev semi-norms.
\begin{lemma}\label{lem:estimPhi}
    Let $p\ge3$ be an integer and $Z$ be a partition such that $h_Z\le 1/(p+1)$.
    For all integers $s$ with $0\leq s\leq 2$,
    there exist functions $\phi_{p,Z}^{(s)} \in S_{p,p-1,Z}$ such that
    \begin{equation}\label{eq:phi}
        \partial^t \phi_{p,Z}^{(s)}(0) = \vardelta_{s,t},\quad
        \partial^t \phi_{p,Z}^{(s)}(1) = 0,
        \quad\mbox{and}\quad 
        |\phi_{p,Z}^{(s)}|_{H^t(0,1)}^2 \le C h_Z^{1+2(s-t)}
    \end{equation}
    for all integers $t$ with $0\leq t\leq 2$.
\end{lemma}
\begin{proof}
    We prove the lemma by explicitly constructing the functions $\phi_{p,Z}^{(s)}$.
    Since we fix one partition $Z$, we just write $h:=h_Z$.
    For integers $\ell$ with $1\le \ell \le3$, define the points
    \begin{equation*}
        \eta_\ell := \min \left\{ z \in Z : z \ge \frac{4\ell ph}{9} \right\}
    \end{equation*}
    and note that the following inequalities hold:
    \begin{equation*}
       0 < \tfrac{4p}{9} h \le \eta_1 \le \tfrac{4p+9}{9} h < \tfrac{8p}{9} h \le \eta_2 \le \tfrac{8p+9}{9}h < \tfrac{12p}{9}h  \le \eta_3 \le \min\{ \tfrac{12p+9}9h, 1\}.
    \end{equation*}
    From these bounds, we deduce 
    \begin{equation}\label{eq:etabound}
            C p h \le \eta_\ell \le C p h,
            \qquad
            |\eta_\ell-\eta_{\ell'}| \ge Cph>0
            \qquad
            \mbox{for}
            \quad \ell \ne \ell'.
    \end{equation}
    We define functions $\psi_{p,\eta}$ via
    \begin{equation*}
        \psi_{p,\eta}(\xi):=(\max\{0,1-\xi/\eta\})^p
    \end{equation*}
    and, using these functions, we define the functions $\phi_{p,Z}^{(s)}$ by
    \begin{equation}\label{eq:phi def}
    \begin{aligned}
        \phi_{p,Z}^{(0)}(\xi)
        &:= \frac{\eta_1^2}{\Delta_{1,2}\Delta_{1,3}} \psi_{p,\eta_1}(\xi) 
        - \frac{\eta_2^2}{\Delta_{1,2}\Delta_{2,3}} \psi_{p,\eta_2}(\xi) 
        + \frac{\eta_3^2}{\Delta_{1,3}\Delta_{2,3}} \psi_{p,\eta_3}(\xi),
        \\
        \phi_{p,Z}^{(1)}(\xi) 
        &:= \frac{1}{p}\left(
        \frac{\eta_1^2(\eta_2+\eta_3)}{\Delta_{1,2}\Delta_{1,3}} \psi_{p,\eta_1}(\xi) 
        -\frac{\eta_2^2(\eta_1+\eta_3)}{\Delta_{1,2}\Delta_{2,3}} \psi_{p,\eta_2}(\xi) 
        + \frac{\eta_3^2(\eta_1+\eta_2)}{\Delta_{1,3}\Delta_{2,3}} \psi_{p,\eta_3}(\xi)
        \right),
        \\
        \phi_{p,Z}^{(2)}(\xi)
        & := \frac{\eta_1\eta_2\eta_3}{p(p-1)}\left( 
        \frac{\eta_1}{\Delta_{1,2}\Delta_{1,3}} \psi_{p,\eta_1}(\xi) 
        - \frac{\eta_2}{\Delta_{1,2}\Delta_{2,3}} \psi_{p,\eta_2}(\xi)
        + \frac{\eta_3}{\Delta_{1,3}\Delta_{2,3}} \psi_{p,\eta_3}(\xi)
        \right),
    \end{aligned}
    \end{equation}
    where $\Delta_{\ell,\ell'}:=\eta_\ell-\eta_{\ell'}$. Those functions satisfy the boundary conditions in \eqref{eq:phi} by construction.

    Applying the triangle inequality and Young's inequality, and using the bounds in \eqref{eq:etabound}, we immediately obtain from \eqref{eq:phi def} that
    \begin{equation}\label{eq:phibound}
        |\phi_{p,Z}^{(s)}|_{H^t(0,1)}^2 \le Ch^{2s} \sum_{\ell=1}^3 |\psi_{p,\eta_\ell}|_{H^t(0,1)}^2,
        \quad\mbox{for}\quad
        0\leq s,t\leq 2.
    \end{equation}
    The (semi-)norms of the functions $\psi_{p,\eta_\ell}$ can be computed explicitly. Combining this with the estimates in \eqref{eq:etabound}, we have
    \begin{equation*}
    \begin{aligned}
        \| \psi_{p,\eta_\ell} \|_{L^2}^2 = \frac{\eta_\ell}{2p+1}& \le C h,\quad
        | \psi_{p,\eta_\ell} |_{H^1}^2 = \frac{p^2}{\eta_\ell(2p-1)} \le C h^{-1},\\ \quad
        | \psi_{p,\eta_\ell} |_{H^2}^2 &= \frac{(p-1)^2p^2}{\eta_\ell^3 (2p-3)} \le C h^{-3}.
    \end{aligned}
    \end{equation*}    
    Substituting these estimates into \eqref{eq:phibound} yields the desired norm bounds stated in \eqref{eq:phi}.
    This completes the proof.
\end{proof}

The following lemma states the existence of a projector that satisfies interpolation and approximation properties like the $H^3$-orthogonal projector from Lemma~\ref{lem:1D_edge_cond} without the requirement $p \ge 5$ from Lemma~\ref{lem:1D_edge_cond} (c).
\begin{lemma}\label{lem:low p projector} 
    Let $p$ and $k$ be integers with $3 \le k+1 \le p$ and let $Z$ be a partition with $h_Z\le 1/(p+1)$.
    There is a linear operator $\Pi_{p,k,Z}^{(*)}:H^3 (0,1) \to S_{p,k,Z}$ such that the following statements hold.\\
    (a) $\Pi_{p,k,Z}^{(*)}$ interpolates at the endpoints, i.e.,
    for all for all integers $s$ with $0\le s \le 2$
    \begin{equation*}
        \partial^s \Pi_{p,k,Z}^{(*)}u(0) = \partial^su(0)
        \quad\mbox{and}\quad
        \partial^s \Pi_{p,k,Z}^{(*)}u(1) = \partial^su(1).
    \end{equation*}
    (b) For $t$ and $s$ integers with $0\le t\le 3\le s \le p+1$
    and all $u\in  H^{\min\{s,k+1\}}(0,1)\cap\mathcal{H}^s(Z;(0,1))$,
    the following error estimate holds:
    \begin{equation*}
        |u-\Pi_{p,k,Z}^{(*)} u|_{H^t(0,1)} \le C(p) h_Z^{s-t} |u|_{\mathcal H^s(Z;(0,1))} .
    \end{equation*}
\end{lemma}
\begin{proof}
    Let $Z$ be an arbitrary but fixed partition of $[0,1]$ with size $h_Z$. \citeRefDetail{Lemma~4.1}{ST20} guarantees
    that 
    there is a  partition
    $\widetilde Z=(\widetilde z_0,\dots, \widetilde z_m)$
    such that 
    $\min\{1,h_Z\} \le \widetilde z_i-\widetilde z_{i-1} \le 3h_Z$, $i=1,\dots,m$
    (quasi-uniformity)
    and $S_{p,k,\widetilde Z}\subseteq S_{p,k,Z}$.
    For each $u\in H^{\min\{s,k+1\}}(0,1)$, let 
    $u_h := \Pi_{p,k,\widetilde Z}^{(2)} u$
    and
    $v_h := \Pi_{\overline p,\overline p-1,\widetilde Z}^{(3)} u$, with $\overline p = \max\{p,5\}$,
    and
    $\Pi_{p,k,Z}^{(*)} u:=w_h \in S_{p,k,Z}$ via
    \begin{equation*}
        w_h(\xi)
        :=
        u_h(\xi)
        +
        (u''(0)-u_h''(0)) \phi^{(2)}_{p,Z}(\xi)
        +
        (u''(1)-u_h''(1)) \phi^{(2)}_{p,\overline Z}(1-\xi),
    \end{equation*}
    where the functions $\phi^{(2)}_{p,Z}$ and $\phi^{(2)}_{p,\overline Z}$ are as in Lemma~\ref{lem:estimPhi}
    and $\overline Z$ is the reverse of the $Z$.
    Using the triangle inequality and the estimates of that Lemma and Lemma~\ref{lem:1d trace} for $\eta=h_Z$, we obtain
    \begin{equation*}
    \begin{aligned}
        &|u-\Pi_{p,k,Z}^{(*)}u|_{H^t(0,1)}
        \\&\quad\le
        |u-u_h|_{H^t(0,1)}
        +
        |u''(0)-u_h''(0)|
        |\phi^{(2)}_{p,Z}|_{H^t(0,1)}
        +
        |u''(1)-u_h''(1)|
        |\phi^{(2)}_{p,\overline Z}|_{H^t(0,1)}
        \\
        &\quad=
        |u-u_h|_{H^t(0,1)}
        +
        |v_h''(0)-u_h''(0)|
        |\phi^{(2)}_{p,Z}|_{H^t(0,1)}
        +
        |v_h''(1)-u_h''(1)|
        |\phi^{(2)}_{p,\overline Z}|_{H^t(0,1)}
        \\
        &\quad\le
        |u-u_h|_{H^t(0,1)}
        +
        C
        h_Z^{5/2-t}
        (
        h_Z^{1/2}
        |v_h-u_h|_{H^{3}(0,1)}
        +
        h_Z^{-1/2}
        |v_h-u_h|_{H^{2}(0,1)}
        ).
    \end{aligned}
    \end{equation*}
    We can further bound $|u-u_h|_{H^t(0,1)} \leq |u-v_h|_{H^t(0,1)} + |v_h-u_h|_{H^t(0,1)}$.
    Observing $v_h-u_h\in S_{\overline p,2,\widetilde Z}$,
    a standard inverse estimate, cf.~\citeRefDetail{Theorem~3.91}{Schwab}, yields
    $|v_h-u_h|_{H^\sigma(0,1)}\le C(p) h_Z^{-\sigma} \|v_h-u_h\|_{L^2(0,1)}$
    for all integers $\sigma$ with $0\le\sigma\le3$.
    Thus, we obtain using the triangle inequality
    \begin{equation*}
    \begin{aligned}
        |u-\Pi_{p,k,Z}^{(*)}u|_{H^t(0,1)} &\leq |u-v_h|_{H^t(0,1)} + C(p)
        h_Z^{-t}\|v_h-u_h\|_{L^2(0,1)} \\
        & \leq |u-v_h|_{H^t(0,1)} + C(p)
        h_Z^{-t}(\|v_h-u\|_{L^2(0,1)} + \|u-u_h\|_{L^2(0,1)}).
    \end{aligned}
    \end{equation*}
    Lemma~\ref{lem:1D_edge_cond} (d) gives the desired upper bound.
    \end{proof}
The above lemma can be combined with Lemma~\ref{lem:1D_edge_cond} to obtain the following.
\begin{corollary}\label{coro:star}
    Let $p$ and $k$ be integers with $3 \le k+1 \le p$ and let $Z$ be a partition with $h_Z\le 1/(p+1)$.
    There is a linear operator $\Pi_{p,k,Z}^{(*)}:H^3 (0,1) \to S_{p,k,Z}$ such that the following statements hold.\\
    (a) $\Pi_{p,k,Z}^{(*)}$ interpolates at the endpoints, i.e.,
    for all for all integers $s$ with $0\le s \le 2$
    \begin{equation*}
        \partial^s \Pi_{p,k,Z}^{(*)}u(0) = \partial^su(0)
        \quad\mbox{and}\quad
        \partial^s \Pi_{p,k,Z}^{(*)}u(1) = \partial^su(1).
    \end{equation*}
    (b) For $t$ and $s$ integers with $0\le t\le 3\le s \le p+1$
    and all $u\in  H^{\min\{s,k+1\}}(0,1)\cap\mathcal{H}^s(Z;(0,1))$,
    the following error estimate holds:
    \begin{equation*}
        |u-\Pi_{p,k,Z}^{(*)} u|_{H^t(0,1)} \le C(p) h_Z^{s-t} |u|_{\mathcal H^s(Z;(0,1))} .
    \end{equation*}
\end{corollary}
\begin{proof}
    For $p\in\{3,4\}$, this follows directly from Lemma~\ref{lem:low p projector} and the dependence on $p$ is not denoted, since $p$ can only have finitely many different values.
    For larger values of $p$, choose
    $\Pi_{p,k,Z}^{(*)}:= \Pi_{p,k,Z}^{(3)}$.
    Then, Lemma~\ref{lem:1D_edge_cond} (b), (c) and~(d) give the desired statements.
\end{proof}

\subsection{Estimates for tensor-product splines}\label{subsec:tensor-estim}
 
We construct a projector into $S_{p,k,\mathbf{Z}}$ for $\mathbf Z=(Z_1,Z_2)$ analogously to \citeOneRef{T18} using the idea of tensor-product projection. 
First, we define the projectors $\widehat\Pi^{(r)}_{p,k,\mathbf Z,j}$, $r\ge0$ integer, $j\in\{1,2,3,4\}$, by
\begin{equation}\label{eq:one-directional-projectors}
\begin{aligned}
    &(\widehat\Pi^{(r)}_{p,k,\mathbf Z,1}u) (\cdot,\xi_2):=(\widehat\Pi^{(r)}_{p,k,\mathbf Z,3}u) (\cdot,\xi_2):=\Pi^{(r)}_{p,k,Z_1} (u(\cdot,\xi_2)) \; \text{ for all }\ \xi_2 \in [0,1],\\
    &(\widehat\Pi^{(r)}_{p,k,\mathbf Z,2}u) (\xi_1,\cdot):=(\widehat\Pi^{(r)}_{p,k,\mathbf Z,4}u) (\xi_1,\cdot):=\Pi^{(r)}_{p,k,Z_2} (u(\xi_1,\cdot)) \; \text{ for all }\ \xi_1 \in [0,1]
\end{aligned}
\end{equation}
and set
\begin{equation}\label{eq:tensor-product-projector}
    \tpprojector := \widehat\Pi^{(r)}_{p,k,\mathbf Z,1} \widehat\Pi^{(r)}_{p,k,\mathbf Z,2}.
\end{equation}
For $u\in H^{r+2}(\widehat \Omega)$, the Sobolev embedding theorem \citeRefDetail{Theorem~4.12, Part II}{Adams}
guarantees $u\in C^r(\widehat\Omega)$. As in \citeOneRef{T18}, these projectors commute
and $\tpprojector$ maps into $S_{p,k,\mathbf{Z}}$, i.e., for all $u\in H^{r+2}(\widehat\Omega)$, we have
\begin{equation}\label{eq:2Dinl}
        \tpprojector u 
        = \widehat\Pi^{(r)}_{p,k,\mathbf Z,1} \widehat\Pi^{(r)}_{p,k,\mathbf Z,2}u 
        = \widehat\Pi^{(r)}_{p,k,\mathbf Z,2} \widehat\Pi^{(r)}_{p,k,\mathbf Z,1} u
        \in S_{p,k,\mathbf{Z}}.
\end{equation}
The following lemma provides the main approximation error estimate for the tensor-product projector $\tpprojector$, which is defined for functions from the bent Sobolev space, cf. \citeOneRef{Bazilevs2006},
\begin{equation*}
    \mathcal{H}^{s,\ell}(\mathbf Z;\widehat\Omega) := \left\{ u\in \mathcal H^s(\mathbf Z;\widehat\Omega) : \; \partial_1^{t_1}\partial_2^{t_2} u \in L^2(\widehat\Omega) \; \mbox{for all} \quad \begin{array}{l}
    (t_1,t_2) \in\{0,\ldots,\ell\}^2 \\
    \mbox{with }t_1+t_2 \leq s
    \end{array}
    \right\},
\end{equation*}
which satisfies $\mathcal{H}^{s,\ell}(\mathbf Z;\widehat\Omega) = H^s(\widehat\Omega)$ if $s\leq \ell$.
\begin{lemma}\label{lem:M2D24}
    Let $p$, $k$, $r$, $t_1$, $t_2$, and $s$ be integers with
    $0\le t_1,t_2 \le r\le k+1\le p$ and
    $2r\leq s \leq p+1$.
    Then, we have
    \begin{equation*}
        \left\|\partial_1^{t_1}\partial_2^{t_2} \left(u- \tpprojector u\right)\right\|_{L^2(\widehat\Omega)} \le C h_{\mathbf Z}^{s-t_1-t_2} |u|_{\mathcal H^s(\mathbf Z;\widehat \Omega)}
        \quad\mbox{for all}\quad u\in \mathcal{H}^{s,k+1}(\mathbf Z;\widehat\Omega).
    \end{equation*}
\end{lemma}
\begin{proof}
    Let $u\in \mathcal{H}^{s,k+1}(\mathbf Z;\widehat\Omega)$ be arbitrary but fixed. To simplify the notation, let $\tpprojectorsimple := \tpprojector$ and $\tpprojectorsimpleI{j} := \widehat \Pi^{(r)}_{p,k,\mathbf{Z},j}$, for $j\in\{1,2\}$. Further, we denote by $\|\cdot\|_{L^2(\mathbf Z;\widehat\Omega)}$ the $\ell_2$-norm of all element-wise $L^2$-norms, i.e.,
    \begin{equation*}
        \|u\|_{L^2(\mathbf Z;\widehat\Omega)} := \left( \sum_{E\in\elemSet_{\mathbf Z}}\| u \|_{L^2(E)}^2 \right)^{1/2}.
    \end{equation*}
    By applying the projector $\Pi^{(r)}_{p,k,Z_1}$ to $u(\cdot,\xi_2)$, Lemma~\ref{lem:1D_edge_cond} (d) implies for all integers $m$, $t'$ and $s'$ with $0\leq t'\leq r\leq s'\leq p+1$ and $0\le m\le s-s'$,
    \begin{multline}\label{eq:lem:M2D24:proof0}
    \| \partial_1^{t'} \partial_2^m (u-\tpprojectorsimpleI{1} u) \|_{L^2(\mathbf Z;\widehat\Omega)}^2
    =\sum_{E\in\elemSet_{\mathbf Z}}\int_{E_2} \| \partial^{t'} (I-\Pi^{(r)}_{p,k,Z_1}) w(\cdot,\xi_2)\|_{L^2(E_1)}^2\mathrm{d}\xi_2\\
    \le C h_{\mathbf Z}^{2(s'-t')} \sum_{E\in\elemSet_{\mathbf Z}}\int_{E_2}|w(\cdot,\xi_2)|^2_{H^{s'}(E_1)}\mathrm{d}\xi_2
    =
    C h_{\mathbf Z}^{2(s'-t')} \| \partial_1^{s'} \partial_2^m u \|_{L^2(\mathbf Z;\widehat\Omega)}^2 ,
    \end{multline}
    where $w:=\partial_2^m u$, element-wise. 
    The condition $t'\ge 2r-p-1$ from Lemma~\ref{lem:1D_edge_cond} (d) follows since we assume $2r\le p+1$.
    By symmetry, we also have
    \begin{equation}\label{eq:lem:M2D24:proof0b}
        \| \partial_1^m \partial_2^{t'}  (u-\tpprojectorsimpleI{2} u) \|_{L^2(\mathbf Z;\widehat\Omega)}
        \le
        C h_{\mathbf Z}^{s'-t'}
        \| \partial_1^m  \partial_2^{s'} u \|_{L^2(\mathbf Z;\widehat\Omega)}.
    \end{equation}
    For any derivative $\partial_{1}^{t_1}\partial_{2}^{t_2}$, with $0\le t_1,t_2\leq r$, we obtain
    using the identity \mbox{$I-\tpprojectorsimpleI{2}\tpprojectorsimpleI{1} = (I-\tpprojectorsimpleI{2}) + (I - \tpprojectorsimpleI{1})-(I-\tpprojectorsimpleI{2})(I-\tpprojectorsimpleI{1})$}, the triangle inequality
    \begin{align}\nonumber 
    \| \partial_{1}^{t_1}\partial_{2}^{t_2} (u - \tpprojectorsimple u) \|_{L^2(\widehat{\Omega})}
    &\leq \ \| \partial_{1}^{t_1}\partial_{2}^{t_2} (I - \tpprojectorsimpleI{2}) u \|_{L^2(\widehat{\Omega})}
     + \| \partial_{1}^{t_1}\partial_{2}^{t_2} (I - \tpprojectorsimpleI{1}) u \|_{L^2(\widehat{\Omega})} 
     \\&\qquad\quad + \| \partial_{1}^{t_1}\partial_{2}^{t_2} (I-\tpprojectorsimpleI{2}) (I - \tpprojectorsimpleI{1}) u \|_{L^2(\widehat{\Omega})}.\label{eq:lem:M2D24:proof5}
    \end{align}
    To estimate the first summand, we use
    \eqref{eq:lem:M2D24:proof0b} with $t':=t_2$, $s':=s-t_1$, and $m:=t_1$
    to obtain
    \begin{equation}\label{eq:lem:M2D24:proof5a}
        \| \partial_{1}^{t_1}\partial_{2}^{t_2} (I - \tpprojectorsimpleI{2}) u \|_{L^2(\widehat{\Omega})} \leq C h_{\mathbf Z}^{s-t_1-t_2} \| \partial_1^{t_1}  \partial_2^{s-t_1} u \|_{L^2(\mathbf Z;\widehat\Omega)}.
    \end{equation}
    For the second one, we
    use~\eqref{eq:lem:M2D24:proof0} with $t':=t_1$, $s':=s-t_2$, and $m:=t_2$ to obtain
    \begin{equation}\label{eq:lem:M2D24:proof5b}
        \| \partial_{1}^{t_1}\partial_{2}^{t_2} (I - \tpprojectorsimpleI{1}) u \|_{L^2(\widehat{\Omega})} \leq C h_{\mathbf Z}^{s-t_1-t_2} \| \partial_1^{s-t_2}  \partial_2^{t_2} u \|_{L^2(\mathbf Z;\widehat\Omega)}.
    \end{equation}
    For the third summand, we first
    use~\eqref{eq:lem:M2D24:proof0b} with $t':=t_2$, $s':=r$, and $m:=t_1$
    and then use~\eqref{eq:lem:M2D24:proof0} with $t':=t_1$, $s':=s-r$, and $m:=r$, to obtain
    \begin{align}\nonumber
        \| \partial_{1}^{t_1}\partial_{2}^{t_2} (I-\tpprojectorsimpleI{2}) (I - \tpprojectorsimpleI{1}) u \|_{L^2(\widehat{\Omega})}
        &\le
        C h_{\mathbf Z}^{r-t_2} \| \partial_{1}^{t_1}\partial_{2}^{r} (I - \tpprojectorsimpleI{1}) u \|_{L^2(\mathbf Z;\widehat{\Omega})}
        \\\label{eq:lem:M2D24:proof5c}
        &\le C h_{\mathbf Z}^{s-t_1-t_2} \| \partial_1^{s-r}  \partial_2^{r} u \|_{L^2(\mathbf Z;\widehat\Omega)}.
    \end{align}
    By substituting
    \eqref{eq:lem:M2D24:proof5a},
    \eqref{eq:lem:M2D24:proof5b}, and
    \eqref{eq:lem:M2D24:proof5c}
    into
    \eqref{eq:lem:M2D24:proof5},
    we obtain
    \begin{align*}
        &\| \partial_{1}^{t_1}\partial_{2}^{t_2} (I - \tpprojectorsimpleI{2}) u \|^2_{L^2(\widehat{\Omega})}\\
        &\quad\leq C h_{\mathbf Z}^{2(s-t_1-t_2)} 
        \left(
            \|\partial_1^{t_1}\partial_2^{s-t_1} u\|_{L^2(\mathbf Z;\widehat\Omega)}^2
            +\|\partial_1^{s-t_2}\partial_2^{t_2} u\|_{L^2(\mathbf Z;\widehat\Omega)}^2
            +\|\partial_1^{s-r}\partial_2^{r} u\|_{L^2(\mathbf Z;\widehat\Omega)}^2
        \right)\\
        &\quad\leq C h_{\mathbf Z}^{2(s-t_1-t_2)} | u |_{\mathcal{H}^s(\mathbf Z;\widehat\Omega)}^2,
    \end{align*}
    which finishes the proof.
\end{proof}

The next lemma establishes that the trace and normal derivatives of $\widehat{\Pi}^{(r,r)}_{p,k,{\mathbf Z}} u$ coincides with the trace and normal derivatives of the correct directional projector of $u$.
\begin{lemma}\label{lem:2D_edge_cond_edge}
    Let $p$, $k$, $r$, and $s$ be integers with $0\le s < r\leq k+1\leq p$ and $p\ge 2r-1$. For all $j\in\{1,2,3,4\}$, we have
    \begin{equation}\label{eq:2D_edge_cond_edge:statement}
        (\mathbf{n}_j\cdot \nabla)^s \left(\widehat{\Pi}^{(r,r)}_{p,k,{\mathbf Z}} u\right)\big|_{\widehat\Sigma_j} 
        = \left(\widehat\Pi^{(r)}_{p,k,\mathbf Z,j} (\mathbf{n}_j\cdot \nabla)^s u \right)\big|_{\widehat\Sigma_j}
        \quad\mbox{for all} \quad u\in H^{r+2}(\widehat \Omega).
    \end{equation}
\end{lemma}
\begin{proof}
    Let $u\in H^{r+2}(\widehat \Omega)$ be arbitrary but fixed and $u_h:=\widehat{\Pi}^{(r,r)}_{p,k,{\mathbf Z}} u$.
    Without loss of generality, we only consider $\widehat{\Sigma}_1 = (0,1)\times\{0\}$. From the definition, we immediately obtain that $u_h(\xi_1,\xi_2) = (\Pi_{p,k,Z_2}^{(r)} w(\xi_1,\cdot))(\xi_2)$, where $w := \widehat\Pi^{(r)}_{p,k,\mathbf Z,1} u$.
    For the normal derivatives on $\widehat\Sigma_1$, we have $\mathbf{n}_1 \cdot \nabla  = -\partial_{2} $ and therefore
    \begin{align*}
        (\mathbf{n}_1 \cdot \nabla)^s u_h(\xi_1,0) 
        = - \left(\partial_2^s \Pi_{p,k,Z_2}^{(r)} w(\xi_1,\cdot)\right)(0)
        &= -\partial_2^s w(\xi_1,0) \\
        &= \left(\widehat\Pi^{(r)}_{p,k,\mathbf Z,1} \left(\mathbf{n}_1 \cdot \nabla \right)^s u\right)(\xi_1,0)
    \end{align*}
    since Lemma~\ref{lem:1D_edge_cond} (b) and (c) guarantees that the projector interpolates the function values and the derivatives at the endpoints
    and because the derivative $\partial_2$ commutes with the projector $\widehat\Pi^{(r)}_{p,k,\mathbf Z,1}$. This shows the second statement in~\eqref{eq:2D_edge_cond_edge:statement}.
    By symmetry, the same arguments apply for the other edges.
\end{proof}
In the next lemma we show that the two-dimensional tensor-product projector interpolates function values and derivatives at the corners of the parameter domain.
\begin{lemma}\label{lem:2D_edge_cond}
    Let $p$, $k$, $r$, $s$, and $t$ be integers with $0\le s,t < r\leq k+1\leq p$ and $p\ge 2r-1$. For all $\widehat{\mathbf{x}}_\ell \in \widehat{\mathcal X}$, we have
    \begin{equation*}
        \partial_{1}^s\partial_{2}^t u(\widehat{\mathbf{x}}_\ell) 
        = \partial_{1}^s\partial_{2}^t (\widehat{\Pi}^{(r,r)}_{p,k,{\mathbf Z}} u)(\widehat{\mathbf{x}}_\ell)
        \quad\mbox{for all} \quad u\in H^{r+2}(\widehat \Omega).
    \end{equation*}
\end{lemma}
\begin{proof}
    Without loss of generality, we only consider $\widehat{\mathbf x}_1 = (0,0)$.
    Using~\eqref{eq:tensor-product-projector} and
    since the differentials commute and since the differentiation in one direction commutes with the projector for the other spatial direction, we obtain by applying Lemma~\ref{lem:2D_edge_cond_edge} for $\widehat \Gamma_1$
    \begin{equation*}
       \partial_1^s \partial_2^t
       (\widehat{\Pi}^{(r,r)}_{p,k,{\mathbf Z}}  u)(0,0)
       =
       \left(
       \partial_1^s \widehat{\Pi}^{(r)}_{p,k,\mathbf Z,1} 
       \left(
       \partial_2^t \widehat{\Pi}^{(r)}_{p,k,\mathbf Z,2}  u
       \right)\right)(0,0)
       =
       \partial_1^s \left(
       \partial_2^t \widehat{\Pi}^{(r)}_{p,k,\mathbf Z,2}  u
       \right)(0,0) 
    \end{equation*}
    and by applying that Lemma for $\widehat \Gamma_4$ further
    \begin{equation*}
       \partial_1^s \partial_2^t
       (\widehat{\Pi}^{(r,r)}_{p,k,{\mathbf Z}}  u)(0,0)
       =
       \left(
       \partial_2^t \widehat{\Pi}^{(r)}_{p,k,\mathbf Z,2} \partial_1^s u
       \right)(0,0)
       =
       \partial_1^s
       \partial_2^t  u(0,0),
    \end{equation*}
    which finishes the proof.
    The result for the other sides follows by symmetry.
\end{proof}

By extending the one-dimensional result from Lemma \ref{lem:1d trace}, the following lemma provides a scaled trace inequality along the edges of the two-dimensional parameter domain.

\begin{lemma}\label{lem:2d trace}
    For all integers $t \ge 0$ and all $j\in  \{1,2,3,4\}$, we have
    \begin{equation*}
    \begin{aligned}
        |u|_{H^t(\widehat \Sigma_j)}
        &\le
        \sqrt{
        2\eta^{-1}
        \|(\mathbf t_j \cdot \nabla)^t u\|_{L^2(\widehat \Omega)}^2
        +
        \eta
        \|(\mathbf t_j \cdot \nabla)^t (\mathbf n_j \cdot \nabla) u\|_{L^2(\widehat \Omega)}^2
        } \\
        &\le
        \sqrt{2}\eta^{-1/2}
        \|(\mathbf t_j \cdot \nabla)^t u\|_{L^2(\widehat \Omega)}
        +
        \eta^{1/2}
        \|(\mathbf t_j \cdot \nabla)^t (\mathbf n_j \cdot \nabla) u\|_{L^2(\widehat \Omega)}
    \end{aligned}
    \end{equation*}
    for all $u\in \mathcal{H}^{t+1,\max\{1,t\}}(\mathbf Z;\widehat\Omega)$ and all $\eta \in (0,1]$.
\end{lemma}
\begin{proof}
    For $\widehat \Sigma_1$, the result follows immediately by applying Lemma~\ref{lem:1d trace} to $\partial_1^{t} u(\xi_1,\cdot)$.
    The result for the other sides follows by symmetry.
\end{proof}

\subsection{Estimates for patch-wise AS-$G^1$ splines}\label{subsec:patchwise-estim}

We restrict ourselves to one patch $\Omega^{(i)}$ and consider the pull-back with respect to its parameterization. Thus, to simplify the notation in the following, we drop the upper index $(i)$ from all quantities that depend on the patch, i.e., the gluing data and directions. From now on, we assume that $p$ and $k$ are integers with $3 \leq k+2 \leq p$ and that $\mathbf Z=(Z_1,Z_2)$ is a partition with $h_{\mathbf Z}\le 1/(p+1)$.
As quasi-interpolation operator, we choose $\widehat \Pi^{(2,2)}_{p,k,\mathbf Z}$, i.e., the projector defined in \eqref{eq:tensor-product-projector} with $r=2$.
That projector
does not satisfy the conditions from Lemma~\ref{lem:continuity_new_notaion} for traces and normal derivatives across interfaces. Therefore, we define a modified projector $\tpprojectorS$ that satisfies these conditions. 
Let $u\in H^4(\widehat\Omega)$ be the function to be approximated. To further simplify the notation, we use the abbreviations
\begin{equation*}
    \tpprojectorsimple  := \widehat{\Pi}^{(2,2)}_{p,k,\mathbf Z} 
    \qquad\mbox{and}\qquad
    \tpprojectorsimpleI{j}  := \widehat{\Pi}^{(2)}_{p,k,\mathbf Z,j} . 
\end{equation*}
Specifically, for $\tpprojectorS$ we modify the function values and directional derivatives along the boundary, compared to $\tpprojectorsimple$. Since those changes have to be extended to the domain $\widehat \Omega$, we first introduce extension operators. For each $j\in  \{1,2,3,4\}$ and each $\sigma\in\{0,1\}$, we define an extension operator $\edgeextensionS{j}{\sigma}$ via
\begin{equation}\nonumber
\begin{aligned}
         &(\edgeextensionS{1}{\sigma} w)(\xi_1,\xi_2) := w(\xi_1,0) \phi_{p,Z_2}^{(\sigma)}(\xi_2), 
         &&(\edgeextensionS{2}{\sigma} w)(\xi_1,\xi_2) := w(1,\xi_2) \phi_{p,\overline Z_1}^{(\sigma)}(1-\xi_1),
         \\
         &(\edgeextensionS{3}{\sigma} w)(\xi_1,\xi_2) := w(\xi_1,1) \phi_{p,\overline Z_2}^{(\sigma)}(1-\xi_2),
         &&(\edgeextensionS{4}{\sigma} w)(\xi_1,\xi_2) := w(0,\xi_2) \phi_{p,Z_1}^{(\sigma)}(\xi_1),
\end{aligned}
\end{equation}
where the functions $\phi_{p,Z_j}^{(\sigma)}$ are as in Lemma~\ref{lem:estimPhi}
and $\overline Z_j$ denotes the reversed partition.

We define the patch-local AS-$G^1$ projector $\tpprojectorS$ by
\begin{equation}\label{eq:ModifProj}
      \tpprojectorS u  :=
           \tpprojectorsimple u
           + \sum_{\sigma=0}^1 \sum_{j=1}^4
           \edgeextension{j}{\sigma}{\edgeprojectorS{j}{\sigma} u-(\mathbf{n}_j\cdot\nabla)^\sigma\tpprojectorsimple u},
\end{equation}
where
\begin{align}
    \label{eq:edge_corrections1}
    \edgeprojectorS{j}{0}u
    &:= \widehat\Pi^{(*)}_{p,k+1,\mathbf Z, j} u,
    \\
    \label{eq:edge_corrections2}
    \edgeprojectorS{j}{1}u
    &:= \alpha_j \,
        \widehat\Pi^{(2)}_{p-1,k,\mathbf Z, j}
            \left( \mathbf{d}_j \cdot \nabla u \right)
        - \beta_j \left(\mathbf{t}_j \cdot \nabla \edgeprojectorS{j}{0} u \right),
\end{align}
where $\widehat\Pi^{(2)}_{p-1,k,\mathbf Z, j}$ and $\widehat\Pi^{(*)}_{p,k+1,\mathbf Z, j}$ are defined in~\eqref{eq:one-directional-projectors} and analogously to that equation, respectively.
Here $\alpha_j$ and $\beta_j$ are the gluing data, $\mathbf{n}_j$, $\mathbf{t}_j$ and $\mathbf{d}_j$ the normal, tangential and crossing directions, respectively, associated to $\widehat\Sigma_j$ as defined in Lemma~\ref{lem:continuity_new_notaion}. The gluing data and crossing direction depend on the respective patch. Since, in this subsection, we consider one fixed patch, we omit the patch index in the notation.

\begin{theorem}
   For all $u\in H^4(\widehat\Omega)$, we have
   \begin{equation*}
        \tpprojectorS u \in S_{p,k,\mathbf Z} .
   \end{equation*}
\end{theorem}
\begin{proof}
    Since \eqref{eq:2Dinl} already states
    $
        \tpprojectorsimple u \in S_{p,k,\mathbf Z},
    $
    it remains to show that the correction terms added in~\eqref{eq:ModifProj} also belong to \(S_{p,k,{\mathbf Z}}\).
    Consider the term $F^{(0)}_1:=\edgeextension{1}{0}{\edgeprojectorS{1}{0} u-\tpprojectorsimple u}$.
    Using the definition of the extension operator, \eqref{eq:edge_corrections1}, \eqref{eq:one-directional-projectors}, and Lemma~\ref{lem:2D_edge_cond_edge}, we have
    \begin{align}\nonumber
        F^{(0)}_1(\xi_1,\xi_2) 
        &= \left(\edgeextension{1}{0}{\edgeprojectorS{1}{0} u-\tpprojectorsimple u}\right)(\xi_1,\xi_2)
        =
        \left(\edgeprojectorS{1}{0} u-\tpprojectorsimple u\right)(\xi_1,0) \phi_{p,Z_2}^{(0)}(\xi_2) \\
        &=
        \left(\Pi^{(*)}_{p,k+1,Z_1} u(\cdot,0)
        -\Pi^{(2)}_{p,k,Z_1} u(\cdot,0)\right)(\xi_1) \phi_{p,Z_2}^{(0)}(\xi_2).
    \label{eq:proof:lem:M2Dincl:1}
    \end{align}
    The definitions of the projectors guarantee
    $\Pi^{(*)}_{p,k+1,Z_1} u(\cdot,0)
        -\Pi^{(2)}_{p,k,Z_1} u(\cdot,0)\in S_{p,k,Z_1}$.
    Since Lemma~\ref{lem:estimPhi} implies $\phi_{p,Z_2}^{(0)}\in S_{p,k,Z_2}$, we obtain $F^{(0)}_1 \in S_{p,k,\mathbf Z}$.
    The other three sides are handled analogously.
  
    Next, consider the term $F_1^{(1)}:=\edgeextension{1}{1}{\edgeprojectorS{1}{1}u-\mathbf{n}_1\cdot\nabla\tpprojectorsimple u}$.
    Analogously to \eqref{eq:proof:lem:M2Dincl:1} and using $\mathbf{n}_1\cdot\nabla = -\partial_2$, we have
    \begin{equation*}
    \begin{aligned}
        F^{(1)}_1(\xi_1,\xi_2)
        &= \left(\edgeextension{1}{1}{\edgeprojectorS{1}{1}u+\partial_2\tpprojectorsimple u}\right)(\xi_1,\xi_2) 
        = \left(\edgeprojectorS{1}{1}u +\partial_2\tpprojectorsimple u\right)(\xi_1,0) \, \phi_{p,Z_2}^{(1)}(\xi_2).
    \end{aligned}
    \end{equation*}
    Recall that, using $\mathbf{t}_1\cdot\nabla = -\partial_1$, we have $\edgeprojectorS{1}{1}u = \alpha_1 \Pi^{(2)}_{p-1,k,\mathbf Z, 1} \left(\mathbf{d}_1\cdot\nabla u\right) + \beta_1 \partial_1 \edgeprojectorS{1}{0} u$
    and observe that $(\Pi^{(2)}_{p-1,k,\mathbf Z, 1} w)(\cdot,0)\in S_{p-1,k,Z_1}$. Further, $(\edgeprojectorS{1}{0}u)(\cdot,0)\in S_{p,k+1,Z_1}$, so we have for the derivative $(\partial_1 \edgeprojectorS{1}{0}u)(\cdot,0)\in S_{p-1,k,Z_1}$. Since $\alpha_1$ and $\beta_1$ are linear functions, taking the product with them yields a function in $S_{p,k,Z_1}$.
    Consequently, $\left(\edgeprojectorS{1}{1}u + \partial_2\tpprojectorsimple u\right)(\cdot,0)\in S_{p,k,Z_1}$. 
    From Lemma~\ref{lem:estimPhi}, we have $\phi_{p,Z_2}^{(1)}\in S_{p,k,Z_2}$. 
    Therefore,
    $
        F^{(1)}_1 \in  S_{p,k,\mathbf Z}.
    $
    The same holds by symmetry for the remaining sides. This completes the proof.
\end{proof}
The next lemma provides an estimate for bounding the Sobolev semi-norm of a product of functions based on the Leibniz rule for the dervative.
\begin{lemma}\label{lem:product rule} 
    Let $s\ge 0$ an integer. The estimate 
    \begin{equation*}
            |vw|_{\mathcal{H}^s(Z;(0,1))} \le 
            C(s) 
            \|v\|_{W^{s,\infty}(0,1)}
            \|w\|_{\mathcal{H}^{s}(Z;(0,1))}
    \end{equation*}
    holds
    for all $v\in W^{s,\infty}(0,1)$ and $w \in \mathcal{H}^{s}(Z;(0,1))$,
    where we use the standard norm
    $\|v\|_{W^{s,\infty}(0,1)}=\max_{m\in\{0,\dots,s\}} \|\partial^m v\|_{L^\infty(0,1)}$.
\end{lemma}
\begin{proof}
    Using the Leibniz rule, the following identity holds:
    \begin{equation*}
    \partial^s (vw) = \sum_{m=0}^s \binom{s}{m} (\partial^m v)(\partial^{s-m} w).
    \end{equation*}
    Taking the \(L^2\)-norm on any element $E\in\elemSet_Z$ and applying the triangle inequality and Hölder's inequality yields
    \begin{equation*}
            |vw|_{H^s(E)} 
            \le 
            \sum_{m=0}^s 
            \binom{s}{m} 
            \|\partial^m v\|_{L^\infty(E)}
            \|\partial^{s-m}w\|_{L^2(E)}
            \le
            C(s)
            \|v\|_{W^{s,\infty}(E)}
            \|w\|_{H^s(E)}.
    \end{equation*}
    Summing over all element norms finishes the proof since $\|v\|_{W^{s,\infty}(E)} \leq \|v\|_{W^{s,\infty}(0,1)}$ and the sum of the binomial coefficients equals $2^s$ and can be absorbed in the constant.
\end{proof}

The following lemma bounds the approximation error associated with the zero-th order boundary correction terms of the modified projector.

\begin{lemma}\label{lem:corrector estimate value}
    Let $t$ and $s$ be integers with $0\leq t\leq 2$ and $4\leq s\leq p+1$.
    For all $j\in  \{1,2,3,4\}$ and all $u\in \mathcal{H}^{s,k+1}(\mathbf Z;\widehat\Omega)$ with $u|_{\widehat\Sigma_j} \in H^{\min\{s,k+2\}}(\widehat\Sigma_j)$, we have 
    \begin{equation*}
        \left|\edgeextension{j}{0}{\edgeprojectorS{j}{0} u-\tpprojectorsimple u}\right|_{H^t(\widehat \Omega)} 
        \le C h_{\mathbf Z}^{s-t} |u|_{\mathcal H^s(\mathbf{Z};\widehat\Omega)}.
    \end{equation*}
\end{lemma}
\begin{proof}
    Consider the edge $\widehat\Sigma_1$. Let $u$ be arbitrary but fixed.
    As in~\eqref{eq:proof:lem:M2Dincl:1}, we have
    \begin{equation}\label{eq:corrector estimate value:proof-1}
        \edgeextension{1}{0}{f^{(0)}_1}(\xi_1,\xi_2) 
        =
        f^{(0)}_1(\xi_1,0)\phi_{p,Z_2}^{(0)}(\xi_2)
        \quad
        \mbox{where}
        \quad
        f^{(0)}_1:=
        \edgeprojectorS{1}{0} u
        -\tpprojectorsimple u.
    \end{equation}  
    For all integers $m$ with $0\le m\le1$, we obtain using~\eqref{eq:edge_corrections1}, Lemma~\ref{lem:2d trace} and by applying Corollary \ref{coro:star} to $u(\cdot,\xi_2)$ and $\partial_2u(\cdot,\xi_2)$, and by integrating over $\xi_2$
    \begin{align}\nonumber
        &|\edgeprojectorS{1}{0} u - u|_{H^{m}({\widehat\Sigma}_1)}
        =
        |\widehat\Pi^{(*)}_{p,k+1,\mathbf Z, 1} u- u|_{H^{m}({\widehat\Sigma}_1)}
        \\\nonumber
        &\qquad\le
        \sqrt 2
        h_{\mathbf Z}^{-1/2}
        \|\partial_1^m(\widehat\Pi^{(*)}_{p,k+1,\mathbf Z, 1} -I) u\|_{L^2({\widehat\Omega})}
        +
        h_{\mathbf Z}^{1/2} 
        \|\partial_2^m (\widehat\Pi^{(*)}_{p,k+1,\mathbf Z, 1}-I) \partial_2 u \|_{L^2({\widehat\Omega})}
        \\
        &\qquad\leq C h_{\mathbf Z}^{s-1/2-m} |u|_{\mathcal H^s(\mathbf{Z}; \widehat{\Omega})}.
        \label{eq:corrector estimate value:proof-2}
    \end{align}
    By applying Lemmas~\ref{lem:M2D24},~\ref{lem:2D_edge_cond_edge} and~\ref{lem:2d trace} to $u-\tpprojectorsimple u$, we analogously obtain
    \begin{align}\nonumber
            |u - \tpprojectorsimple u|_{H^{m}({\widehat\Sigma}_1)} 
            &\le\sqrt 2
            h_{\mathbf Z}^{-1/2}
            \|\partial_1^m (I - \tpprojectorsimple) u\|_{L^2({\widehat\Omega})}
            +
            h_{\mathbf Z}^{1/2}
            \|\partial_1^m\partial_2(I-\tpprojectorsimple) u\|_{L^{2}({\widehat\Omega})}
            \\
            &\leq C h_{\mathbf Z}^{s - 1/2 - m} |u|_{\mathcal H^s(\mathbf{Z}; \widehat{\Omega})}.
            \label{eq:corrector estimate value:proof-3}
    \end{align}
    Combining~\eqref{eq:corrector estimate value:proof-1}, the triangle inequality, \eqref{eq:corrector estimate value:proof-2} and \eqref{eq:corrector estimate value:proof-3} yields 
    \begin{equation}\label{eq:fi bound}
        |f_1^{(0)}|_{H^{m}({\widehat\Sigma}_1)}
        \leq 
        |\edgeprojectorS{1}{0} u - u|_{H^{m}({\widehat\Sigma}_1)}
        +|u-\tpprojectorsimple u|_{H^{m}({\widehat\Sigma}_1)}
        \leq
        C h_{\mathbf Z}^{s - 1/2 - m} |u|_{\mathcal H^s(\mathbf{Z}; \widehat{\Omega})}.
    \end{equation}
    Using the product rule for Sobolev spaces on rectangles, Lemma~\ref{lem:estimPhi} and~\eqref{eq:fi bound}, we immediately obtain 
    \begin{equation*}
    \begin{aligned}
        &\left|\edgeextension{1}{0}{\edgeprojectorS{1}{0} u-\tpprojectorsimple u}\right|_{H^t(\widehat{\Omega})}^2 
        = \sum_{m=0}^t |f_1^{(0)}|_{H^m({\widehat\Sigma}_1)}^2 |\phi_{p,Z_2}^{(0)}|_{H^{t-m}(0,1)}^2 
        \\
        &\quad\le \sum_{m=0}^t \left( C h_{\mathbf Z}^{s-1/2-m} |u|_{\mathcal H^s(\mathbf{Z}; \widehat{\Omega})} \right)^2 \left( Ch_{\mathbf Z}^{1/2 + m - t} \right)^2
        \le Ch_{\mathbf Z}^{2(s-t)} |u|_{\mathcal H^s(\mathbf{Z}; \widehat{\Omega})}^2,
    \end{aligned}
    \end{equation*}
    which finishes the proof for $\widehat \Sigma_1$. The proof for the other cases can be done analogously.
\end{proof}

Analogously, the next lemma provides an approximation error estimate for the first order boundary correction terms.
\begin{lemma}\label{lem:corrector estimate derivative}
    Let $t$ and $s$ be integers with $0\leq t\leq 2$ and $4\leq s\leq p+1$.
    For all $j\in  \{1,2,3,4\}$ and all $u\in \mathcal{H}^{s,k+1}(\mathbf Z;\widehat\Omega)$ with $u|_{\widehat\Sigma_j} \in H^{\min\{s,k+2\}}(\widehat\Sigma_j)$, we have 
    \begin{equation*}
        \left|\edgeextension{j}{1}{\edgeprojectorS{j}{1} u-\mathbf{n}_j \cdot \nabla \tpprojectorsimple u}\right|_{H^t(\widehat \Omega)} 
        \le C(\mathbf G,s) h_{\mathbf Z}^{s-t}\|u\|_{\mathcal H^s(\mathbf{Z};\widehat\Omega)}.
    \end{equation*}
\end{lemma}
\begin{proof}
    Consider the edge $\widehat\Sigma_1$ and note that $\mathbf{n}_1=(0,-1)$ and $\mathbf{t}_1=(-1,0)$. Let $u$ be arbitrary but fixed. As above, we have
    \begin{equation*}
        \edgeextension{1}{1}{f^{(1)}_1}(\xi_1,\xi_2)
        =
        f^{(1)}_1(\xi_1,0)
        \phi_{p,Z_2}^{(1)}(\xi_2),
        \quad
        \mbox{where}
        \quad
        f^{(1)}_1 := \edgeprojectorS{1}{1} u+\partial_2 \tpprojectorsimple u.
    \end{equation*}
    Using \eqref{eq:edge_corrections2}, \eqref{eq:edge_corrections1}, and \eqref{eq:partial d def},
    we obtain
    \begin{equation*}
    \begin{aligned}
        f^{(1)}_1
        &=
        \alpha_1 \widehat\Pi^{(2)}_{p-1,k,\mathbf Z,1}\left(\mathbf{d}_1 \cdot \nabla u \right) + \beta_1 \partial_1 \edgeprojectorS{1}{0} u + \partial_2 \tpprojectorsimple u \\
        &=
        \alpha_1 \left(
        \widehat\Pi^{(2)}_{p-1,k,\mathbf Z,1}g_1 
        -
        g_1\right)
        +
        \alpha_1 g_1
        + \beta_1 \partial_1 \widehat\Pi_{p,k+1,\mathbf Z, 1}^{(*)} u 
        + \partial_2 \tpprojectorsimple u\\
        &=
        \alpha_1 \left(
        \widehat\Pi^{(2)}_{p-1,k,\mathbf Z,1}g_1 
        -
        g_1\right)
        + \beta_1\partial_1  \left(\widehat\Pi_{p,k+1,\mathbf Z, 1}^{(*)} u - u\right)
        + \partial_2 \left(\tpprojectorsimple u - u\right),
    \end{aligned}
    \end{equation*}
    where
    \begin{equation*}
        g_1(\xi_1,\xi_2)
        := 
        \frac{-\partial_2 u(\xi_1,\xi_2)-\beta_1(\xi_1) \partial_1 u(\xi_1,\xi_2)}{\alpha_1(\xi_1)}.
    \end{equation*}
    Using the triangle inequality, we have for all integers $m$ with $0\le m\le 2$
    \begin{align}\nonumber
    \left|f_1^{(1)}\right|_{H^{m}(\widehat\Sigma_1)}
    &\le
    \left|\alpha_1 \left( \widehat\Pi^{(2)}_{p-1,k,\mathbf{Z},1}g_1 - 
        g_1 \right)\right|_{H^{m}(\widehat\Sigma_1)}
    +\left|\beta_1\partial_1 (\widehat\Pi^{(*)}_{p,k+1,\mathbf Z,1}u-u)\right|_{H^{m}(\widehat\Sigma_1)}
    \\&\qquad+\left|\partial_2(\tpprojectorsimple u-u)\right|_{H^{m}(\widehat\Sigma_1)}.\label{eq:proof:corrector estimate derivative-1-new}
    \end{align}
    To estimate the first term, we use Lemma~\ref{lem:product rule} with
    $\|\alpha_1\|_{W^{m,\infty}(\widehat\Sigma_1)}\le C(\mathbf G,s)$,
    Lemma~\ref{lem:2d trace} with $\eta=h_{\mathbf Z}$, Lemma~\ref{lem:1D_edge_cond}~(d)
    and again Lemma~\ref{lem:product rule} with $\|\alpha_1^{-1}\|_{W^{s-1,\infty}(\widehat\Sigma_1)}\le C(\mathbf G,s)$ and $\|\alpha_1^{-1}\beta_1 \|_{W^{s-1,\infty}(\widehat\Sigma_1)}\le C(\mathbf G,s)$ to obtain 
    \begin{equation}\label{eq:proof:corrector estimate derivative-2-new}
    \begin{aligned}
        &\left|\alpha_1 \left( \widehat\Pi^{(2)}_{p-1,k,\mathbf{Z},1}g_1 - 
        g_1 \right)\right|_{H^{m}(\widehat\Sigma_1)}
        \le C(\mathbf G,s) \|\widehat\Pi^{(2)}_{p-1,k,\mathbf{Z},1}g_1 - 
        g_1 \|_{H^{m}(\widehat\Sigma_1)}
        \\&\quad
        \le
        C(\mathbf G,s)\sum_{\sigma=0}^m\sum_{\sigma'=0}^1
        h_{\mathbf Z}^{\sigma'-1/2}
        \|\partial_1^{\sigma} (\Pi^{(2)}_{p-1,k,\mathbf Z,1}-I)\partial_2^{\sigma'} g_1 \|_{L^{2}(\widehat\Omega)}
        \\
        &\quad
        \le
        C(\mathbf G,s) h_{\mathbf Z}^{s-3/2-m}
        |g_1|_{\mathcal{H}^{s-1}(\mathbf{Z};\widehat\Omega)}
        \\&\quad
        \le
        C(\mathbf G,s) h_{\mathbf Z}^{s-3/2-m}
        \left(
        |\alpha_1^{-1} \partial_2 u|_{\mathcal{H}^{s-1}(\mathbf{Z};\widehat\Omega)}
        +
        |\alpha_1^{-1} \beta_1 \partial_1 u|_{\mathcal{H}^{s-1}(\mathbf{Z};\widehat\Omega)}
        \right)
        \\
        &\quad
        \le
        C(\mathbf G,s) h_{\mathbf Z}^{s-3/2-m}
        \|u\|_{\mathcal{H}^{s}(\mathbf{Z};\widehat\Omega)}.
    \end{aligned}
    \end{equation}
    To estimate the second term in~\eqref{eq:proof:corrector estimate derivative-1-new}, we note that $\partial_1$ is a derivative along the edge. By applying
    Lemma~\ref{lem:product rule} with
    $\|\beta_1\|_{W^{m,\infty}(\widehat\Sigma_1)}\le C(\mathbf G,s)$,
    Lemma~\ref{lem:2d trace} and Corollary~\ref{coro:star} (b), we obtain 
    \begin{equation}\label{eq:proof:corrector estimate derivative-3-new}
     \begin{aligned}
        &\left|\beta_1\partial_1 (\widehat\Pi^{(*)}_{p,k+1,\mathbf Z,1}u-u)\right|_{H^{m}(\widehat\Sigma_1)}
        \le
        C(\mathbf{G},s)
         \left\|\partial_1(\widehat\Pi^{(*)}_{p,k+1,\mathbf Z, 1}u - u)\right\|_{H^m(\widehat\Sigma_1)} \\
        &\qquad\le C(\mathbf{G},s) \sum_{\sigma=1}^{m+1}\sum_{\sigma'=0}^{1}
        h_{\mathbf Z}^{\sigma'-1/2}
        \|\partial_1^{\sigma}(\widehat\Pi^{(*)}_{p,k+1,\mathbf Z, 1}-I)\partial_2^{\sigma'} u\|_{L^{2}({\widehat\Omega})} 
        \\
        &\qquad
        \le C(\mathbf G,s) h_{\mathbf Z}^{s-3/2-m} |u|_{\mathcal{H}^{s}(\mathbf{Z};\widehat\Omega)}.
    \end{aligned}
    \end{equation}
    For the third term in~\eqref{eq:proof:corrector estimate derivative-1-new}, the commuting property from Lemma~\ref{lem:2D_edge_cond_edge} implies that $\partial_2 \tpprojectorsimple u = \tpprojectorsimpleI{1} \partial_2 u$ on $\widehat\Sigma_1$. 
    Moreover, we observe $\partial_2 \tpprojectorsimpleI{1}u=\tpprojectorsimpleI{1} \partial_2 u$.
    By applying Lemma~\ref{lem:2d trace} to the error, we obtain
    \begin{equation*}
    \begin{aligned}
        &\left|\partial_2(\tpprojectorsimple u-u)\right|_{H^{m}(\widehat\Sigma_1)}
        = |(\tpprojectorsimpleI{1} -I)\partial_2 u|_{H^{m}(\widehat\Sigma_1)} \\
        &\qquad\le \sqrt{2}h_{\mathbf Z}^{-1/2} \|\partial_1^m(\tpprojectorsimpleI{1} -I)\partial_2 u\|_{L^{2}(\widehat\Omega)} + h_{\mathbf Z}^{1/2} \|\partial_1^m(\tpprojectorsimpleI{1} -I)\partial_2^2 u\|_{L^{2}(\widehat\Omega)}.
    \end{aligned}
    \end{equation*}
    Using Lemma~\ref{lem:1D_edge_cond} (d) for the functions $\partial_2 u$ and $\partial_2^2u$, we arrive at
    \begin{equation}\label{eq:proof:corrector estimate derivative-4-new}
    \begin{aligned}
        \left|\partial_2(\tpprojectorsimple u-u)\right|_{H^{m}(\widehat\Sigma_1)}
        &\le C h_{\mathbf Z}^{s-3/2-m} |u|_{\mathcal{H}^s(\mathbf{Z};\widehat\Omega)}.
    \end{aligned}
    \end{equation}

    By combining~\eqref{eq:proof:corrector estimate derivative-1-new}, \eqref{eq:proof:corrector estimate derivative-2-new}, \eqref{eq:proof:corrector estimate derivative-3-new} and~\eqref{eq:proof:corrector estimate derivative-4-new}, we have
    \begin{equation}\label{eq:fi1 bound-new}
        |f_1^{(1)}|_{H^{m}({\widehat\Sigma}_1)} \leq C(\mathbf G,s) h_{\mathbf Z}^{s-3/2 - m} \|u\|_{\mathcal{H}^s(\mathbf{Z};\widehat\Omega)}.
    \end{equation}
    Finally, we apply the product rule for Sobolev spaces on rectangles,
    Lemma~\ref{lem:estimPhi} and the bound \eqref{eq:fi1 bound-new} to obtain 
    \begin{equation*}
        \left|\edgeextension{1}{1}{f^{(1)}_1}\right|_{H^t(\widehat{\Omega})}^2 
        = 
        \sum_{m=0}^t
        |f_1^{(1)}|_{H^m({\widehat\Sigma}_1)}^2 |\phi_{p,Z_2}^{(1)}|_{H^{t-m}(0,1)}^2
        \leq C(\mathbf G,s) h_{\mathbf Z}^{2(s-t)} \|u\|_{\mathcal{H}^s(\mathbf{Z};\widehat\Omega)}^2.
    \end{equation*}
    This finishes the proof for $\widehat\Sigma_1$. The proof for the other edges is analogous.
\end{proof}

This theorem is the main approximation error estimate for the modified tensor-product projector. 

\begin{theorem}\label{thm:M2D24} Let $t$ and $s$ be integers with $0\leq t\leq 2$ and $4\leq s\leq p+1$. For all $u\in \mathcal{H}^{s,k+1}(\mathbf Z;\widehat\Omega)$, with $u|_{\Sigma_j} \in H^{\min\{s,k+2\}}(\Sigma_j)$ for all $j\in  \{1,2,3,4\}$, we have 
\begin{equation*}
    \|u-\tpprojectorS u\|_{H^t(\widehat \Omega)} \le C(\mathbf G,s) h_{\mathbf Z}^{s-t} \|u\|_{\mathcal H^s(\mathbf{Z};\widehat\Omega)}.
\end{equation*}
\end{theorem}
\begin{proof}
    For all integers $m$ with $0\leq m\leq 2$, the estimate
    $|u-\tpprojectorS u|_{H^{m}(\widehat \Omega)}^2 \le C(\mathbf G,s) h_{\mathbf Z}^{2(s-m)} \|u\|_{\mathcal H^s(\mathbf{Z};\widehat\Omega)}^2$ follows from the definition of the projector, i.e.,~\eqref{eq:ModifProj}, and Lemmas~\ref{lem:M2D24}, \ref{lem:corrector estimate value} and~\ref{lem:corrector estimate derivative}.
    Since $h_{\mathbf Z}\le 1$, the estimate for the norm follows immediately by taking the sum for $m=0,\dots,t$.
\end{proof}

The following lemma states that the extensions from different edges do not interfere with each other. This is characterized by conditions on the traces and normal derivatives of the correction terms. 
\begin{lemma}\label{lem:no-interference}
    Let $u\in H^4(\widehat\Omega)$ be arbitrary but fixed and define
    for $j\in  \{1,2,3,4\}$ and $\sigma\in\{0,1\}$
    \begin{equation*}
        f_j^{(\sigma)}:=\edgeprojectorS{j}{\sigma}  u - \left(\mathbf{n}_j\cdot\nabla\right)^\sigma \tpprojectorsimple u.
    \end{equation*}
    We then have
    \begin{align}
        \label{eq:no-interference:2}
         \edgeextension{j}{0}{f_j^{(0)}} \big|_{\widehat\Sigma_{j}} &= f_j^{(0)} \big|_{\widehat\Sigma_{j}} & \mbox{ and }\quad
                \mathbf{n}_j\cdot\nabla \edgeextension{j}{0}{f_j^{(0)}} \big|_{\widehat\Sigma_{j}} 
        &= 0,
        \\
        \label{eq:no-interference:3}
        \edgeextension{j}{1}{f_j^{(1)}} \big|_{\widehat\Sigma_{j}} &= 0 & \mbox{ and }\quad
                \mathbf{n}_j\cdot\nabla \edgeextension{j}{1}{f_j^{(1)}} \big|_{\widehat\Sigma_{j}} 
        &= f_j^{(1)} \big|_{\widehat\Sigma_{j}},
    \end{align}
    as well as
    \begin{equation}
        \label{eq:no-interference:1}
        \edgeextension{j}{\sigma}{f_j^{(\sigma)}} \big|_{\widehat\Sigma_{j'}} =
        \mathbf{n}_{j'}\cdot\nabla 
        \edgeextension{j}{\sigma}{f_j^{(\sigma)}} \big|_{\widehat\Sigma_{j'}} 
        = 0
    \end{equation}
    for all $j' \in\{1,2,3,4\}\setminus\{j\}$.
\end{lemma}
\begin{proof}
    Consider the case $j=1$. We have
    \begin{equation*}
            \edgeextension{1}{\sigma}{f_1^{(\sigma)}}(\xi_1,\xi_2)
            =
            f_1^{(\sigma)}(\xi_1,0)
            \phi^{(\sigma)}_{p,Z_2}(\xi_2).
    \end{equation*}
    Lemma~\ref{lem:estimPhi} states that $\partial^{\sigma'}\phi^{(\sigma)}_{p,Z_2}(1)=0$, for $\sigma'\in\{0,1\}$, which shows~\eqref{eq:no-interference:1} for $j'=3$, i.e., the edge opposite to $\widehat \Sigma_1$. Moreover, it states $\partial^{\sigma'}\phi^{(\sigma)}_{p,Z_2}(0)=\vardelta_{\sigma,\sigma'}$ for $\sigma, \sigma'\in\{0,1\}$,
    where $\vardelta_{\sigma,\sigma'}$ is the Kronecker-delta, which shows~\eqref{eq:no-interference:2} and~\eqref{eq:no-interference:3}, i.e., the proper interpolation on the edge $\widehat \Sigma_1$ itself.
    
    In order to show~\eqref{eq:no-interference:1} for $j'\in\{2,4\}$, consider the case $\sigma=0$ first. We have using~\eqref{eq:edge_corrections1} and Lemma~\ref{lem:2D_edge_cond_edge},
    and using~\eqref{eq:one-directional-projectors}
    \begin{equation*}
    \begin{aligned}
        f_1^{(0)}(\cdot,0)
        =
        \left(\widehat\Pi_{p,k+1,\mathbf Z, 1}^{(*)} u\right)(\cdot,0)
        -
        \left(\tpprojectorsimpleI{1} u\right)(\cdot,0)
        =
        \widehat\Pi_{p,k+1,Z_1}^{(*)} u(\cdot,0)
        -
        \widehat\Pi_{p,k,Z_1}^{(2)} u(\cdot,0).
    \end{aligned}
    \end{equation*}
    Corollary~\ref{coro:star} (a) and
    Lemma~\ref{lem:1D_edge_cond} (b) and (c),
    state that these two projectors interpolate function values and (at least) first derivatives on the boundary, so we have
    $f_1^{(0)}(0,0)=\partial_{1}f_1^{(0)}(0,0)
    = f_1^{(0)}(1,0)=\partial_{1}f_1^{(0)}(1,0) = 0$, which shows~\eqref{eq:no-interference:1} for $j'\in\{2,4\}$, $\sigma'\in\{0,1\}$ and $\sigma=0$.
    
    For the case $\sigma=1$, we first note that we have, using $\mathbf{n}_1=(0,-1)$, $\mathbf{t}_1=(-1,0)$, $\mathbf{d}_1= \frac{1}{\alpha_1}\left(\mathbf{n}_1 + \beta_1 \mathbf{t}_1 \right)=(-\beta_1/\alpha_1,-1/\alpha_1)$,~\eqref{eq:edge_corrections1}, and~\eqref{eq:edge_corrections2}, 
    \begin{equation*}
    \begin{aligned}
        f_1^{(1)}(\cdot,0)
        &= \left(\alpha_1 \,
        \widehat\Pi^{(2)}_{p-1,k,\mathbf Z, 1}
            \left( \mathbf{d}_1 \cdot \nabla u \right)
        + \beta_1 \partial_1 \edgeprojectorS{j}{0} u + \partial_2 \tpprojectorsimple u\right)(\cdot,0)  \\
        &  =
        \alpha_1 \,
        \Pi^{(2)}_{p-1,k, Z_1}
        \left(\frac{-w-\beta_1 \partial_{1} v}{\alpha_1}\right)
        + \beta_1 \partial_{1} \Pi^{(*)}_{p,k+1,Z_1}v
        + \Pi^{(2)}_{p,k,Z_1} w.
    \end{aligned}
    \end{equation*}
    Here and in what follows, we use
    $v:=u(\cdot,0)$ and $w:=\partial_2 u(\cdot,0)$.
    Corollary~\ref{coro:star} (a) and
    Lemma~\ref{lem:1D_edge_cond} (b) and (c) yield
    $f_1^{(1)}(0,0)=f_1^{(1)}(1,0)=0$, which shows
    \begin{equation}\label{eq:proof-no-interference-110}
    \edgeextension{1}{1}{f_1^{(1)}} \big|_{\widehat\Sigma_{j'}} = 0
    \end{equation}
    for $j'\in\{2,4\}$.
    By taking the derivative, we have 
    \begin{align}\nonumber
        \partial_1
        f_1^{(1)}
        &=
        (\partial_1\alpha_1) \,
        \Pi^{(2)}_{p-1,k,\mathbf Z, 1}
        \left(\frac{-w-\beta_1 \partial_{1} v}{\alpha_1}\right)
        +
        \alpha_1 \,
        \partial_1
        \Pi^{(2)}_{p-1,k,\mathbf Z, 1}
        \left(\frac{-w-\beta_1 \partial_{1} v}{\alpha_1}\right)
        \\&\qquad
        + (\partial_{1}\beta_1)\partial_{1} \Pi^{(*)}_{p,k+1,\mathbf Z, 1} v
        + \beta_1\partial_{1}^2 \Pi^{(*)}_{p,k+1,\mathbf Z, 1} v
        + \partial_{1} \tpprojectorsimpleI{1} w.\label{eq:proof-no-interference-4}
    \end{align}
    Again, using Corollary~\ref{coro:star} (a) and
    Lemma~\ref{lem:1D_edge_cond} (b) and (c) and standard rules for differentiation, we have 
    $\partial_{1}f_1^{(1)}(0,0)=\partial_{1}f_1^{(1)}(1,0)=0$, which, together with~\eqref{eq:proof-no-interference-110}, shows~\eqref{eq:no-interference:1} for $j'\in\{2,4\}$ and $\sigma=1$. This finishes the proof.
\end{proof}

\begin{lemma}\label{lem:M2Dedge_cond}
    The operator $\tpprojectorS$ interpolates on the edges $\widehat\Sigma_j$, $j\in  \{1,2,3,4\}$, in the sense that
    \begin{equation*}
    (\tpprojectorS u)\big|_{\widehat \Sigma_j} 
    = (\edgeprojectorS{j}{0} u) \big|_{\widehat \Sigma_j} 
    \quad\mbox{and}\quad
    (\mathbf{n}_j \cdot \nabla \tpprojectorS u)\big|_{\widehat \Sigma_j} 
    = (\edgeprojectorS{j}{1} u)\big|_{\widehat \Sigma_j} 
    \quad\mbox{for all}\quad u\in H^4(\widehat\Omega).
    \end{equation*}
\end{lemma}
\begin{proof}
    Using~\eqref{eq:ModifProj} and Lemma~\ref{lem:no-interference}, we obtain
    \begin{equation*}
      (\tpprojectorS u)|_{\widehat \Sigma_j}
      =
      (\tpprojectorsimple u)|_{\widehat \Sigma_j}
      + 
      \left(\edgeprojectorS{j}{0} u-\tpprojectorsimple u\right)|_{\widehat \Sigma_j}
      =
      \edgeprojectorS{j}{0} u,
    \end{equation*}
    as well as
    \begin{equation*}
      (\mathbf{n}_j \cdot \nabla \tpprojectorS u)|_{\widehat \Sigma_j}
      =
      (\mathbf{n}_j \cdot \nabla \tpprojectorsimple u)|_{\widehat \Sigma_j}
      + 
      \left(\edgeprojectorS{j}{1} u-\mathbf{n}_j \cdot \nabla\tpprojectorsimple u\right)|_{\widehat \Sigma_j}
      =
      \edgeprojectorS{j}{1} u.
    \end{equation*}
    This finishes the proof.
\end{proof}

\begin{lemma}\label{lem:M2Dnormals}
    The operator $\tpprojectorS$ interpolates directional derivatives in direction $\mathbf{d}_j$ on the edges $\widehat\Sigma_j$, $j\in  \{1,2,3,4\}$ in the sense that
    \begin{equation*}
        \left(
        \mathbf{d}_j \cdot \nabla \tpprojectorS u
        \right)\big|_{\widehat \Sigma_j}
        =
        \Pi_{p-1,k,\mathbf Z,j}^{(2)}\left( \mathbf{d}_j \cdot \nabla u\right)\big|_{\widehat \Sigma_j}
        \quad\mbox{for all}\quad u\in H^4(\widehat\Omega).
    \end{equation*}
\end{lemma}
\begin{proof} 
    Recalling the definition of $\mathbf{d}_j$ in~\eqref{eq:partial d def} and Lemma~\ref{lem:M2Dedge_cond}, we have 
    \begin{equation*}    
    \begin{aligned}
        \left(
        \mathbf{d}_j \cdot \nabla  \tpprojectorS u
        \right)\big|_{\widehat \Sigma_j}
        &=
        \frac{1}{\alpha_j} 
        \left(
        \mathbf{n}_j \cdot \nabla  \tpprojectorS u
        \right)\big|_{\widehat \Sigma_j}
        +
        \frac{\beta_j}{\alpha_j}
        \left(
        \mathbf{t}_j \cdot \nabla  \tpprojectorS u
        \right)\big|_{\widehat \Sigma_j} \\
        &=
        \frac{1}{\alpha_j} 
        \left(
        \edgeprojectorS{j}{1} u
        \right)\big|_{\widehat \Sigma_j}
        +
        \frac{\beta_j}{\alpha_j}
        \left(
        \mathbf{t}_j \cdot \nabla  \edgeprojectorS{j}{0} u
        \right)\big|_{\widehat \Sigma_j}
        .
    \end{aligned}
    \end{equation*}
    Using~\eqref{eq:edge_corrections2}, we further have
    \begin{align*}
        &\left(
        \mathbf{d}_j \cdot \nabla  \tpprojectorS u
        \right)\big|_{\widehat \Sigma_j} \\
        &\quad=
        \frac{1}{\alpha_j} 
        \left(
        \alpha_j \,
        \widehat\Pi^{(2)}_{p-1,k,\mathbf Z,j}
            \left( \mathbf{d}_j \cdot \nabla u \right)
        - \beta_j \mathbf{t}_j \cdot \nabla  \edgeprojectorS{j}{0} u
        \right)\Big|_{\widehat \Sigma_j}
        +
        \frac{\beta_j}{\alpha_j}
        \left(
        \mathbf{t}_j \cdot \nabla  \edgeprojectorS{j}{0} u
        \right)\big|_{\widehat \Sigma_j}
        \\
        &\quad=
        \widehat\Pi^{(2)}_{p-1,k,\mathbf Z, j}
            \left( \mathbf{d}_j \cdot \nabla u \right)\big|_{\widehat\Sigma_j},
    \end{align*}
    which finishes the proof.
\end{proof}

The following lemma states that the operator $\tpprojectorS$ interpolates the function values and derivatives up to order two at the corners of $\widehat\Omega$.
\begin{lemma}\label{lem:M2Dcorner_cond}
     For all integers $s$ and $t$ with $0\leq s+t\leq 2$, all corners $\widehat{\mathbf x}_\ell$, $\ell\in\{1,2,3,4\}$, we have
    \begin{equation*}
        \partial_{1}^s\partial_{2}^t (\tpprojectorS u)(\widehat{\mathbf x}_\ell)
        =\partial_{1}^s\partial_{2}^t u(\widehat{\mathbf x}_\ell)
         \quad\mbox{for all}\quad u\in H^4(\widehat\Omega).
    \end{equation*}
\end{lemma}
\begin{proof}
    Without loss of generality, we consider only the corner $\widehat{\mathbf x}_1=(0,0)$.
    For the case $t=0$, we make use of $\widehat{\mathbf x}_1\in\widehat \Sigma_1$ and Lemma~\ref{lem:M2Dedge_cond} and Corollary~\ref{coro:star} (a) to obtain
    \begin{equation*}
        \partial_{1}^s(\tpprojectorS u)(\widehat{\mathbf x}_1)
        =
        \partial_{1}^s(\edgeprojectorS{j}{0}u)(\widehat{\mathbf x}_1)
        =
        \left(\partial_{1}^s\Pi^{(*)}_{p,k+1,Z_1}u(\cdot,0)\right)(0)
        =
        \partial_{1}^s u(\widehat{\mathbf x}_1)
    \end{equation*}
    for $0\leq s \leq 2$. 
    The case $s=0$, $1\leq t \leq 2$ can be handled analogously by utilizing $\widehat{\mathbf x}_1\in\widehat \Sigma_4$.
    What remains is the case $s=1$, $t=1$. Using again $\widehat{\mathbf x}_1\in\widehat \Sigma_1$ and Lemma~\ref{lem:M2Dedge_cond}, we obtain
    \begin{equation*}
        (\partial_2 \tpprojectorS u)\big|_{\widehat \Sigma_1} 
        =-(\mathbf{n}_1 \cdot \nabla \tpprojectorS u)\big|_{\widehat \Sigma_1} 
        = -(\edgeprojectorS{1}{1} u)\big|_{\widehat \Sigma_1}.
    \end{equation*}
    The proof of Lemma~\ref{lem:no-interference}, in particular~\eqref{eq:proof-no-interference-4}, states that $\partial_1 f_1^{(1)}(0,0) = 0$, where $f_1^{(1)}=\edgeprojectorS{1}{1}  u + \partial_2 \tpprojectorsimple u$. Thus, we get using Lemma~\ref{lem:2D_edge_cond}
    \begin{equation*}
        (\partial_1\partial_2 \tpprojectorS u)(0,0)
        = -(\partial_1\edgeprojectorS{1}{1} u)(0,0) = (\partial_1\partial_2 \tpprojectorsimple u)(0,0) = (\partial_1\partial_2 u)(0,0),
    \end{equation*}    
    which is the desired result. 
    The proof for the other three corners is analogous.
\end{proof}

\subsection{Global compatibility of the multi-patch AS-$\boldsymbol{G^1}$ projector}\label{subsec:global-compat}

The global AS-$G^1$ projector $\tpprojectorSph$ is defined for all $u\in \mathcal H^4(\Omega)$ patch-wise via
\begin{equation}\label{eq:global_projector}
        (\tpprojectorSph u)|_{\Omega^{(i)}}
        := \left(\tpprojectorS ^{(i)}(u\circ \mathbf G^{(i)})\right) \circ (\mathbf G^{(i)})^{-1}
        \quad\mbox{for all}\quad i \in \mathcal I_{\Omega},
\end{equation}
where $\tpprojectorS ^{(i)}$ is the patch-local projector $\tpprojectorS $ from the last subsection using the vectors of breakpoints and the gluing data for $\Omega^{(i)}$. On the interfaces between the patches, the values of $\tpprojectorSph u$ are defined by continuous extension of its values on the neighboring patches. This requires that the limits from both sides agree, which the following lemma provides.
\begin{lemma}\label{lem:mapping}
    For all $u\in \mathcal H^4(\Omega)$, the function $\tpprojectorSph u$ can be continuously extended to the interfaces. For $\tpprojectorSph u$ being the so exteded operator, we have
    \begin{equation}\label{eq:lem:mapping:esimate}
        \tpprojectorSph u \in \mathcal V^1_{AS}
    \end{equation}
    and
    \begin{align}
        \label{eq:lem:mapping:cond-1}
        u\in C^2(\mathbf x)
        &\quad \Rightarrow \quad
        \tpprojectorSph u\in C^2(\mathbf x)
        && \mbox{for any} \quad \mathbf x = \mathbf{G}^{(i)}(\widehat{\mathbf{x}}_j),
        \\
        \label{eq:lem:mapping:cond-2}
        u\big|_{\Sigma} = c
        &\quad \Rightarrow \quad
        (\tpprojectorSph u)\big|_{\Sigma} = c
        && \mbox{for any} \quad \Sigma= \mathbf{G}^{(i)}(\widehat{\Sigma}_j),
        \\  
        \label{eq:lem:mapping:cond-3}
        \mathfrak{d}^{(i)}_j\cdot\nabla u\big|_{\Sigma} = 0
        &\quad \Rightarrow \quad
        \mathfrak{d}^{(i)}_j \cdot\nabla(\tpprojectorSph u)\big|_{\Sigma} = 0
        &&\mbox{for any} \quad \Sigma = \mathbf{G}^{(i)}(\widehat{\Sigma}_j),
    \end{align}
    where $(i,j)\in\mathcal{I}_\Omega\times\{1,2,3,4\}$ and $c\in\mathbb{R}$ is an arbitrary constant.
\end{lemma}
\begin{proof} 
    Let $u\in \mathcal H^4(\Omega)$ be arbitrary but fixed and observe that the Sobolev embedding theorem guarantees $u\in C^1(\Omega)$.
    For each $i \in \mathcal I_\Omega$, we denote by $\widehat u^{(i)} := u\circ \mathbf G^{(i)}$ the pull-back of $u$ to the parameter domain
    and by $\widehat u_h^{(i)} := \tpprojectorS ^{(i)} \widehat u^{(i)}\in S_{p,k,\mathbf Z^{(i)}}$
    its approximation. Observe that the construction~\eqref{eq:global_projector} yields
    $\widehat u_h^{(i)} = (\tpprojectorSph u) \circ \mathbf G^{(i)}$
    in the interior of the parameter domain.
    In order to establish that $\tpprojectorSph u$ can be continuously extended to the interfaces,
    consider one arbitrary interface $\Sigma=\mathbf G^{(i)}(\widehat\Sigma_j)=\mathbf G^{(\neighbori(i,j))}(\widehat\Sigma_{\neighborj(i,j)})$,
    where $\neighbori:=\neighbori(i,j)$ and $\neighborj:=\neighborj(i,j)$ to simplify the notation. 
    Using the pull-back principle~\eqref{eq:pullback-traces}, Lemma~\ref{lem:M2Dedge_cond}, \eqref{eq:one-directional-projectors}, and~\eqref{eq:edge-Z-def}, we obtain
    \begin{equation*}
    \widehat u_h^{(i)}|_{\widehat \Sigma_j}  
    =
    (\widehat\Pi_{AS}^{(i)} \widehat u^{(i)})|_{\widehat \Sigma_j} 
    =
    \widehat \Pi^{(*)}_{p,k+1,Z^{(i)}_j} (\widehat u^{(i)}|_{\widehat \Sigma_j} ) \in S_{p,k+1, Z^{(i)}_j}
    \end{equation*}
    and analogously
    $\widehat u_h^{(\neighbori)}\circ e_{(i,j)}
    =
    \widehat \Pi^{(*)}_{p,k+1,Z^{(i)}_j} (\widehat u^{(\neighbori)}\circ e_{(i,j)})$.
    Since $u\in C^0(\Omega)$ guarantees $\widehat u^{(i)}|_{\widehat \Sigma_j}=\widehat u^{(\neighbori)}\circ e_{(i,j)}$, we immediately obtain
    \begin{equation}\label{eq:proof:lem:mapping:1}
    \widehat u_h^{(i)}|_{\widehat \Sigma_j} 
    =
    \widehat u_h^{(\neighbori)}\circ e_{(i,j)}
    \in S_{p,k+1,Z^{(i)}_j}.
    \end{equation}
    This shows that there is a continuous extension of $\tpprojectorSph u$ to the interface $\Sigma$.
    In the following, $\tpprojectorSph u$ is the so extended function.
    From~\eqref{eq:proof:lem:mapping:1}, the first condition of~\eqref{eq:V1 AS}
    and, since $\widehat \Pi^{(*)}_{p,k+1,Z^{(i)}_j}c=c$ for all constants $c$, also~\eqref{eq:lem:mapping:cond-2} follows.

    In order to establish also the continuity of the derivatives, we observe that~\eqref{eq:directions}, the pull-back principle~\eqref{eq:pullback-d-derivatives}, Lemma~\ref{lem:M2Dnormals}, \eqref{eq:one-directional-projectors}, and~\eqref{eq:edge-Z-def} yield
    \begin{equation}\label{eq:proof:lem:mapping:4}
    \begin{aligned}
    \left((\mathfrak d^{(i)}_j\cdot\nabla\tpprojectorSph u)|_{\overline{\Omega^{(i)}}}\right)\circ \mathbf G^{(i)}|_{\widehat \Sigma_j}  
    &=
    (\mathbf d_j^{(i)}\cdot\nabla \widehat\Pi_{AS}^{(i)} \widehat u^{(i)})|_{\widehat \Sigma_j}
    \\&= 
    \Pi^{(2)}_{p-1,k,Z^{(i)}_j} (\mathbf d_j^{(i)}\cdot\nabla \widehat u^{(i)}|_{\widehat \Sigma_j})
    \in 
    S_{p-1,k,Z^{(i)}_j},
    \end{aligned}
    \end{equation}
    where $w|_{\overline{\Omega^{(i)}}}$ denotes the continuous extension of a function $w$ from $\Omega^{(i)}$ to its closure.
    Since $\mathfrak d^{(i)}_j = -\mathfrak d^{(\neighbori)}_{\neighborj}$, we also have
    \begin{equation*}
    \left((\mathfrak d^{(i)}_j \cdot\nabla\tpprojectorSph u)|_{\overline{\Omega^{(\neighbori)}}}\right)\circ \mathbf G^{(\neighbori)} \circ e_{(i,j)}
    = \Pi^{(2)}_{p-1,k,Z^{(i)}_j} (-\mathbf d^{(\neighbori)}_{\neighborj}\cdot\nabla \widehat u^{(\neighbori)}\circ e_{(i,j)}).
    \end{equation*}
    Since $u\in C^2(\Omega)$ guarantees 
    $\mathbf d_j^{(i)}\cdot\nabla \widehat u^{(i)}|_{\widehat \Sigma_j}=-\mathbf d^{(\neighbori)}_{\neighborj}\cdot\nabla \widehat u^{(\neighbori)}\circ e_{(i,j)}$,
    we obtain
    \begin{equation*}
    \left((\mathfrak d^{(i)}_j\cdot\nabla\tpprojectorSph u)|_{\overline{\Omega^{(i)}}}\right)\circ \mathbf G^{(i)}|_{\widehat \Sigma_j}  
    =
    \left((\mathfrak d^{(i)}_j \cdot\nabla\tpprojectorSph u)|_{\overline{\Omega^{(\neighbori)}}}\right)\circ \mathbf G^{(\neighbori)} \circ e_{(i,j)}
    \in S_{p-1,k,Z^{(i)}_j}.
    \end{equation*}
    This shows the second condition of \eqref{eq:V1 AS} and that $\mathfrak d^{(i)}_j\cdot\nabla \Pi_{AS}  u$ is continuous across interfaces. Since also $\Pi_{AS} u$ is continuous across interfaces and since Definition~\ref{def:ASG1domainAsmpt} guarantees that $\mathfrak d^{(i)}_j$ does not coincide with the tangential direction of $\Sigma$, this guarantees that $\Pi_{AS} u\in C^1(\Omega)$, which finishes the proof of~\eqref{eq:lem:mapping:esimate}.
    In order to show~\eqref{eq:lem:mapping:cond-3}, assume $\mathfrak d^{(i)}_j\cdot\nabla u|_{\Sigma} = 0$. Using Lemma~\ref{lem:continuity_new_notaion}, we know that the pull-back $\mathbf d^{(i)}_j\cdot\nabla (u\circ \mathbf G^{(i)})|_{\widehat \Sigma_j}$ vanishes. Then, \eqref{eq:proof:lem:mapping:4} and $\widehat \Pi^{(2)}_{p-1,k,Z^{(i)}_j}0=0$ gives~\eqref{eq:lem:mapping:cond-3}.
    
    To finally show~\eqref{eq:lem:mapping:cond-1}, assume $u \in C^2(\mathbf{x})$ for some vertex $\mathbf{x} = \mathbf{G}^{(i)}(\widehat {\mathbf x}_j)$. Lemma~\ref{lem:M2Dcorner_cond} ensures that the local projector interpolates the values and derivatives up to second order of the pullback $u \circ \mathbf{G}^{(i)}$ at the corners of the parameter domain. Since the geometry parameterization $\mathbf{G}^{(i)}$ is $C^2$ smooth there, the chain rule implies that the physical derivatives of $\tpprojectorSph u$ at $\mathbf{x}$ are uniquely determined and match those of $u$.
    This completes the proof.
\end{proof}

In order to give estimates of the quasi interpolation error, we use the following lemma.
\begin{lemma}\label{lem:norm:equivalence}
    For all $u\in H^s(\Omega^{(i)})$, $s\ge 0$ integer, such that $\mathbf G^{(i)}|_E\in C^s(E)$, for all $E\in\elemSet_{\mathbf{Z}}$, one has $u\circ \mathbf G^{(i)} \in \mathcal H^s(\mathbf{Z};\widehat\Omega)$ and
    \begin{equation*}
        \|u\|_{H^s(\Omega^{(i)})}
        \le C(\mathbf G,s)
        \|u\circ \mathbf G^{(i)}\|_{\mathcal H^s(\mathbf{Z};\widehat\Omega)},
        \quad
        \|u\circ \mathbf G^{(i)}\|_{\mathcal H^s(\mathbf{Z};\widehat\Omega)}
        \le C(\mathbf G,s)
        \|u\|_{H^s(\Omega^{(i)})}.
    \end{equation*}
    Moreover, for $s=1$, we have an equivalence of $H^1$-seminorms.
\end{lemma}
\begin{proof}
    This estimate follows immediately from the definition of the involved norms, standard chain and substitution rules, and Hölder's inequality, cf. \citeOneRef{Bazilevs2006}.
\end{proof}

\subsection{Proof of the main theorem}\label{subsec:proof}

Now, we are able to prove Theorem~\ref{thrm:main}.
\begin{proof}
    The projector $\tpprojectorSph$ is as defined in~\eqref{eq:global_projector} and continuously extended to the interfaces between the patches as in Lemma~\ref{lem:mapping}. That lemma guarantees
    $\tpprojectorSph u \in \mathcal V_{AS}^1$.
    The combination of Theorem~\ref{thm:M2D24} and Lemma~\ref{lem:norm:equivalence} gives 
    \begin{equation*}
        \|u-\tpprojectorSph u\|_{H^t(\Omega^{(i)})}^2
        \le
        C(\mathbf G,s)
        h_{\mathbf Z^{(i)}}^{2(s-t)}
        \|u\|_{H^s(\Omega^{(i)})}^2.
    \end{equation*}
    The error estimate of the form~\eqref{eq:thrm:main:estimate} follows by taking the sum over all patches and the observation $h_{\mathbf Z^{(i)}}\le C(\mathbf G,s) \physH$.
    The remaining statements~\eqref{eq:thrm:main:cond-1}, \eqref{eq:thrm:main:cond-2}, and~\eqref{eq:thrm:main:cond-3} follow immediately from Lemma~\ref{lem:mapping}.
\end{proof}

\section{Conclusions}\label{sec:conclusions}

We have proven $p$-robust approximation error estimates for $C^1$-smooth isogeometric spaces over planar analysis-suitable $G^1$ multi-patch domains. This extends the approximation error estimates in \citeRefDetail{Theorem~1}{Kapl2021} from bilinear to general AS-$G^1$ parameterizations. Moreover, it verifies the numerical results obtained, e.g., in \citeRefDetail{Section~8.2}{CST16}, \citeRefDetail{Section~4.2}{kapl2018construction}, \citeRefDetail{Section~5}{kapl2019isogeometric} and \citeRefDetail{Section~6}{kapl2019smai}. Since the dimension of the isogeometric space $\dim(\mathcal{V}^{1}_{AS} )$ is of order $O(h^{-2})$ for $h\rightarrow 0$, the $n$-width theory, cf. \citeOneRef{floater2019optimal}, implies that the estimates for maximal Sobolev order $s=p+1$ are asymptotically optimal. Thus, the estimates yield optimal convergence rates for conforming discretizations of fourth-order problems, such as the biharmonic equation and Kirchhoff--Love plate formulations. Moreover, bounds that are independent of $p$ are needed in the convergence analysis of iterative solvers when robustness with respect to both mesh size and polynomial degree is desired. This includes multigrid methods and IETI-DP domain decomposition.

The constructed projection operator can be extended to multi-patch surfaces, which allows us to study approximation properties for isogeometric discretizations over surfaces, where all patch neighborhoods are planarizable. It is also possible to modify the patch-wise projectors to extend the error estimates to analyze adaptive methods, e.g., based on hierarchical or truncated hierarchical splines, which will result in a loss of $p$-robustness. Note that the presented analysis requires that the parameterization is AS-$G^1$, i.e., that linear gluing data exists. For  multi-patch parameterizations that are not AS-$G^1$, the $C^1$-smooth subspace of the $C^0$ multi-patch space lacks approximation power. This reduction leads to a reduced convergence order, or in the worst case to a complete $C^1$-locking that prevents convergence altogether, cf. \citeRefDetail{Section~8.3}{CST16} and \citeRefDetail{Section~4.2}{kapl2018construction}. An in-depth analysis of the dependence of the approximation estimates on the patch parameterizations is of interest for future studies. One could formally establish lower-order convergence rates or prove suboptimality for non-AS-$G^1$ configurations. Further potential extensions include the study of approximation error bounds for discretizations of higher-order continuity over so-called bilinear-like multi-patch parameterizations, and for $C^1$-smooth spaces over volumetric domains.

\section*{Acknowledgments}
This research was funded in whole or in part by the Austrian Science Fund (FWF) 10.55776/P37177 and 10.55776/P33956. This support is gratefully acknowledged.

\appendix
\section{Computing the gluing data}\label{ap:comp-gluing}

We consider a single interface $\Sigma$ shared by two patches $\Omega^{(1)}$ and $\Omega^{(2)}$, as in Figure~\ref{fig:geometric}, where $i=1$ and $\neighbori(1,4)=2$. We define the pullbacks of the isogeometric function $u$ as $\hat{u}^{(1)} = u \circ \mathbf{G}^{(1)}$ and $\hat{u}^{(2)} = u \circ \mathbf{G}^{(2)}$, then the graph of the function is defined as
    \begin{equation*}
    \begin{bmatrix} \mathbf{G}^{(i)} \\ \widehat{u}_h^{(i)} \end{bmatrix} : \widehat{\Omega} \to \mathbb{R}^3.
\end{equation*}
Let $\xi$ denote the parameter along the interface. Thus, the $C^0$-smoothness of the geometry and isogeometric function can be written as
\begin{equation}\nonumber
\mathbf G^{(1)}(0,\xi)=\mathbf G^{(2)}(\xi,0)
\quad \mbox{and}\quad
\widehat{u}_h^{(1)}(0,\xi)=\widehat{u}_h^{(2)}(\xi,0).
\end{equation}
Assuming a $C^0$-matching and $2$-regular multi-patch parameterization,
$C^1$-continuity of the function $u$ is equivalent to $G^1$-continuity of the graph, cf. \citeOneRef{groisser2015matched}. This requires the directional derivative along the interface and the two cross-directional derivatives to be linearly dependent:
\begin{equation*}
    \det \begin{bmatrix} 
        \partial_2 \mathbf{G}^{(1)}(0,\xi) & \partial_1 \mathbf{G}^{(1)}(0,\xi) & \partial_2 \mathbf{G}^{(2)}(\xi,0) \\ 
        \partial_2 \widehat{u}_h^{(1)}(0,\xi)  & \partial_1 \widehat{u}_h^{(1)}(0,\xi)  & \partial_2 \widehat{u}_h^{(2)}(\xi,0)
    \end{bmatrix} = 0.
\end{equation*}
By expanding this determinant along the third row, we obtain the expression
\begin{equation*} 
D_3(\xi)\,\partial_2 \widehat{u}_h^{(1)}(0,\xi) + 
D_2(\xi)\,\partial_1 \widehat{u}_h^{(1)}(0,\xi) 
+ D_1(\xi)\,\partial_2 \widehat{u}_h^{(2)}(\xi,0) 
=0,
\end{equation*}
where
\begin{equation*}
\begin{aligned}
       D_1(\xi) &:= \det \begin{bmatrix} \partial_1 \mathbf{G}^{(1)}(0,\xi) & \partial_2 \mathbf{G}^{(1)}(0,\xi) \end{bmatrix} >0\\
       D_2(\xi) &:= \det \begin{bmatrix} \partial_1 \mathbf{G}^{(2)}(\xi,0) & \partial_2 \mathbf{G}^{(2)}(\xi,0) \end{bmatrix} >0\\
       D_3(\xi) &:= -\det \begin{bmatrix} \partial_1 \mathbf{G}^{(1)}(0,\xi) & \partial_2 \mathbf{G}^{(2)}(\xi,0)\end{bmatrix}.
\end{aligned}
\end{equation*}
Here, we made use of the fact that the $C^0$-smoothness condition immediately implies
$\partial_2 \mathbf G^{(1)}(0,\xi)=\partial_1 \mathbf G^{(2)}(\xi,0)$. The parameterization is analysis-suitable $G^1$ if and only if there exists a function $\gamma:[0,1]\rightarrow \mathbb{R}^+$ and linear functions $\alpha_1$, $\alpha_2$, $\beta_1$, $\beta_2$, called \emph{gluing data}, such that
\[
 D_1 = \alpha_1 \gamma, \quad D_2 = \alpha_2 \gamma, \quad \mbox{and}\quad D_3 = (\alpha_1\beta_2+\alpha_2\beta_1) \gamma.
\]
In order to compute normalized gluing data from a given AS-$G^1$ parameterization, we solve the following minimization problems:
\begin{equation*}
    \max_{\xi \in [0,1]}\max\{\alpha_1(\xi),\alpha_2(\xi)\}
    \to \min \quad \mbox{subject to} \quad \begin{cases}D_1\alpha_2=D_2\alpha_1 \quad \mbox{and} \\\min_{\xi \in [0,1]}\min\{\alpha_1(\xi),\alpha_2(\xi)\}=1
    \end{cases}
\end{equation*}
and
\begin{equation*}
    \max_{\xi \in [0,1]}\max\{|\beta_1(\xi)|,|\beta_2(\xi)|\}\to \min \quad \mbox{subject to} \quad D_1\beta_2+D_2\beta_1=D_3.
\end{equation*}
If $D_1=cD_2$ with $c\in \mathbb{R}^+$, then the minimization yields constant functions $\alpha_1$ and $\alpha_2$. In that case or if $\beta_1$ and $\beta_2$ are constant, then the space $\mathcal{V}^1_{AS}$ in~\eqref{eq:V1 AS} is a proper subspace of $\mathcal{V}^1_{h}$ in~\eqref{eq:Vh1}. Since these two spaces have the same approximation power and similar dimension, either of them can be realized numerically as convenient. Furthermore, in an implementation it is usually not feasible to satisfy the linear constraints of the minimization problems exactly. Thus, a least squares approach can be employed, which yields constant $\alpha_1$ and $\alpha_2$, if they satisfy the constraints within a prescribed tolerance.

If $\mathbf{G}^{(1)}$ and $\mathbf{G}^{(2)}$ are bilinear, one obtains non-normalized gluing data with $\gamma\equiv 1$, which can be normalized directly. Thus, in case of bilinear patches the parameterization is always AS-$G^1$. If the given parameterization is not AS-$G^1$, one can still compute linear gluing data, which can be used for an AS-$G^1$ reparameterization, with the exception of a few degenerate cases, cf. \citeOneRef{kapl2018construction}. Linear gluing data functions are constructed via interpolation:
\begin{equation*}
\begin{aligned}
\alpha_1(\xi) &= D_1(0) (1 - \xi) + D_1(1) \xi, \\
\alpha_2(\xi) &= D_2(0) (1 - \xi) + D_2(1) \xi, \\
\beta_1(\xi) &= b_1^0 (1 - \xi) + b^1_1 \xi, \qquad\qquad b_1^0, b^1_1 \in \mathbb{R}, \\
\beta_2(\xi) &= b_2^0 (1 - \xi) + b^1_2 \xi, \qquad\qquad b_2^0, b^1_2 \in \mathbb{R},
\end{aligned}
\end{equation*}
where the coefficients $b^0_1$, $b^1_1$, $b^0_2$, $b^1_2$ are determined by solving the minimization problem
\begin{equation*}
\int_0^1 \big| D_3 - \underbrace{(\alpha_1\beta_2 + \alpha_2\beta_1)}_{\beta} \big|^2 \mathrm d \xi + \lambda_{\beta} \left( \int_0^1 |\beta_1|^2 \mathrm d \xi + \int_0^1 |\beta_2|^2 \mathrm d \xi \right) \to \min_{b^0_1, b^1_1, b^0_2, b^1_2},
\end{equation*}
with respect to the linear constraints $\beta(0)=D_3(0)$, and $\beta(1)=D_3(1)$. Note that this computation does not reproduce the correct gluing data in case of an AS-$G^1$ parameterization.

While, in general, the gluing data depends on the parameterization in a non-trivial way, in case of bilinear patches or AS-$G^1$ reparameterizations, the constant $C(\mathbf{G},s)$ in Definition~\ref{def:constants} is easily computable and depends only on $\underline{c}$ in~\eqref{eq:lower-bound-Jacobi-det1} and on the $W^{1,\infty}$-norm of the patch parameterizations.

\bibliographystyle{ws-m3as}
\bibliography{references}

\end{document}